\documentclass[11pt]{amsart}
\usepackage[centertags]{amsmath}
\usepackage{amsfonts}
\usepackage{amssymb}
\usepackage{amsthm}
\usepackage{newlfont}
\usepackage{amscd}
\usepackage{mathrsfs}
\usepackage[all,cmtip]{xy}
\usepackage{tikz}
\usepackage[pagebackref=false,colorlinks]{hyperref}
\hypersetup{pdffitwindow=true,linkcolor=blue,citecolor=blue,urlcolor=cyan}
\usepackage{cite}
\usepackage{fancyhdr}
\usepackage{stmaryrd}
\usepackage{listings}
\usepackage[OT2,OT1]{fontenc}
\newcommand\cyr{%
\renewcommand\rmdefault{wncyr}%
\renewcommand\sfdefault{wncyss}%
\renewcommand\encodingdefault{OT2}%
\normalfont
\selectfont}
\DeclareTextFontCommand{\textcyr}{\cyr}

\newcommand{\Mod}[1]{\ (\text{mod}\ #1)}

\usepackage[toc, page] {appendix}

\numberwithin{equation}{section}

\newtheorem{thm}{Theorem}[section]
\newtheorem{cor}[thm]{Corollary}
\newtheorem{lem}[thm]{Lemma}
\newtheorem{prop}[thm]{Proposition}

\newtheorem{conj}[thm]{Conjecture}

\theoremstyle{definition}
\newtheorem{defn}[thm]{Definition}
\newtheorem{choice}[thm]{Choice}
\newtheorem{assu}[thm]{Assumption}
\newtheorem{rem}[thm]{Remark}

\newtheorem{ques}[thm]{Question}

\newcommand{\ks}{\boldsymbol{\kappa}}
\newcommand{\es}{\mathbf{c}}
\newcommand{\sha}{\textrm{{\cyr SH}}}

\begin{document}
\title[Kato's Euler systems for the additive reduction case]{Remarks on Kato's Euler systems for elliptic curves with additive reduction}
\author[C.-H. Kim]{Chan-Ho Kim}
\address{(C.-H. Kim) School of Mathematics, Korea Institute for Advanced Study, 85 Hoegiro, Dongdaemun-gu, Seoul 02455, Republic of Korea}
\email{chanho.math@gmail.com}
\author[K. Nakamura]{Kentaro Nakamura}
\address{(K. Nakamura) Department of Mathematics, Saga University, 1 Honjo-machi, Saga 840-8502, Japan}
\email{nkentaro@cc.saga-u.ac.jp}
\date{\today}
\begin{abstract}
Extending the former work for the good reduction case, we provide a numerical criterion to verify a large portion of the ``Iwasawa main conjecture without $p$-adic $L$-functions" for elliptic curves with \emph{additive} reduction at an odd prime $p$ over the cyclotomic $\mathbb{Z}_p$-extension.
We also deduce the corresponding $p$-part of the Birch and Swinnerton-Dyer formula for elliptic curves of rank zero from the same numerical criterion. We give some explicit examples at the end and specify our choice of Kato's Euler system in the appendix.
\end{abstract}
\maketitle


\section{Introduction}
\subsection{Overview}
This article is a generalization of the numerical criterion for the verification of the Iwasawa main conjecture for modular forms at a good prime \cite{kks} to the additive reduction case.
For the multiplicative reduction case, the main conjecture follows from the good ordinary case and the use of Hida theory (\hspace{1sp}\cite{skinner-pacific}).
This criterion also has an application to the $p$-part of the Birch and Swinnerton-Dyer formula for elliptic curves of rank zero
(c.f.~\cite[Theorem 7.2.1]{jetchev-skinner-wan}).
Since there are only finitely many bad reduction primes for elliptic curves, the criterion can be practically used to check the \emph{full} Birch and Swinnerton-Dyer formula for an elliptic curve of rank zero and not necessarily square-free conductor.
All the known non-CM examples have square-free conductors (\hspace{1sp}\cite[Appendix]{wan-main-conj-ss-ec}).
We put some numerical examples of elliptic curves with additive reduction at the end.

Let $p$ be an odd prime and $E$ be an elliptic curve of conductor $N$ over $\mathbb{Q}$. 
Throughout this article, we assume that $E$ has additive reduction at $p$. In other words, $p^2$ divides $N$.

The construction of Kato's Euler systems \cite{kato-euler-systems} and the formulation of the Iwasawa main conjecture without $p$-adic $L$-functions \cite[Conjecture 12.10]{kato-euler-systems} are insensitive to the reduction type of elliptic curves.
See \cite{jacinto-L_p-additive} for the $\mathbf{D}_{\mathrm{dR}}$-valued $p$-adic $L$-function of elliptic curves with additive reduction and also \cite{delbourgo-compositio} and \cite{delbourgo-jnt} for different attempts to understand the additive reduction case.

We expect that the reader has some familiarity with \cite{rubin-book} and \cite{mazur-rubin-book}.

\subsection{Working assumptions and Kurihara numbers}
Let
\begin{itemize}
\item $\mathrm{Tam}(E)$ be the Tamagawa number of $E$,
\item $N_{\mathrm{st}}$ be the product of split multiplicative reduction primes of $E$, and
\item $N_{\mathrm{ns}}$  be the product of non-split multiplicative reduction primes of $E$.
\end{itemize}
We assume the following conditions throughout this article.
\begin{assu}[Working assumptions]  \label{assu:working_assumptions} $ $
\begin{enumerate}
\item $p$ does not divide ${\displaystyle \mathrm{Tam}(E) \cdot \prod_{\ell \mid N_{\mathrm{st}}} (\ell -1) \cdot \prod_{\ell \mid N_{\mathrm{ns}}} (\ell +1) }$. (c.f. Remark \ref{rem:exceptional_case}.(4)).
\item The mod $p$ Galois representation $\overline{\rho} : G_{\mathbb{Q}} \to \mathrm{Aut}_{\mathbb{F}_p}(E[p])$ is surjective. (Thus, $E$ is non-CM.)
\item The Manin constant is prime to $p$. (It is expected to be true always.)
\end{enumerate}
\end{assu}

Let $f \in S_2(\Gamma_0(N))$ be the newform attached to $E$ by \cite[Theorem A]{bcdt}.
For $\dfrac{a}{b} \in \mathbb{Q}$, we define $\left[\dfrac{a}{b}\right]^+$ by
$$ 2 \pi \int^{\infty}_0 f \left( \frac{a}{b} + iy \right) dy = \left[\dfrac{a}{b}\right]^+ \cdot \Omega^+_E + 
 \left[\dfrac{a}{b}\right]^- \cdot \Omega^-_E$$
 where $\Omega^{\pm}_E$ are the N\'{e}ron periods of $E$. Then it is well-known that $\left[\dfrac{a}{b}\right]^+ \in \mathbb{Q}$.
The following theorem due to G. Stevens yields the $p$-integrality of the value.
\begin{thm}[{\hspace{1sp}\cite[$\S$3]{glenn-stickelberger}}] \label{thm:integrality_modular_symbols}
Under (2) and (3) of Assumption \ref{assu:working_assumptions}, we have
$\left[\dfrac{a}{b}\right]^+ \in \mathbb{Z}_{(p)}$ for $\dfrac{a}{b} \in \mathbb{Q}$.
\end{thm}

A prime $\ell$ is a \textbf{Kolyvagin prime for $(E, p)$} if $(\ell ,Np) =1$, $\ell \equiv 1 \Mod{p}$, and $a_\ell(E) \equiv \ell + 1 \Mod{p}$.
We define the \textbf{Kurihara number for $(E,p)$ at $n$} by
$$\widetilde{\delta}_n := \sum_{a \in (\mathbb{Z}/n\mathbb{Z})^\times} \bigg( \overline{ \left[ \frac{a}{n} \right]^+} \cdot  \prod_{\ell \vert n} \overline{ \mathrm{log}_{\mathbb{F}_\ell} (a) }  \bigg) \in \mathbb{F}_p$$
where $n$ is the square-free product of Kolyvagin primes, $\overline{ \left[ \frac{a}{n} \right]^+}$ is the mod $p$ reduction of $\left[ \frac{a}{n} \right]^+$, and $\overline{ \mathrm{log}_{\mathbb{F}_\ell} (a) }$ is the mod $p$ reduction of the discrete logarithm of $a$ modulo $\ell$ with a fixed primitive root modulo $\ell$. The number $\widetilde{\delta}_n$ itself is not well-defined, but its non-vanishing question is well-defined.

\subsection{The Iwasawa main conjecture \`{a} la Kato}
Let $\mathbb{Q}_\infty$ be the cyclotomic $\mathbb{Z}_p$-extension of $\mathbb{Q}$ and $\mathbb{Q}_n$ be the subextension of $\mathbb{Q}$ in $\mathbb{Q}_\infty$ of degree $p^n$.
Let $\Lambda := \mathbb{Z}_p \llbracket \mathrm{Gal}(\mathbb{Q}_{\infty}/\mathbb{Q})\rrbracket$ be the Iwasawa algebra.
Let $T = \mathrm{Ta}_pE$ be the $p$-adic Tate module of $E$ and $j : \mathrm{Spec}(\mathbb{Q}_n) \to \mathrm{Spec}(\mathcal{O}_{\mathbb{Q}_n}[1/p])$ be the natural map.
Then we define the \textbf{$i$-th Iwasawa cohomology of $E$} by 
$$\mathbb{H}^i(T) := \varprojlim_{n} \mathrm{H}^i_{\mathrm{\acute{e}t}}( \mathrm{Spec}(\mathcal{O}_{\mathbb{Q}_n}[1/p]), j_*T) $$
where $\mathrm{H}^i_{\mathrm{\acute{e}t}}( \mathrm{Spec}(\mathcal{O}_{\mathbb{Q}_n}[1/p]), j_*T)$ is the \'{e}tale cohomology group. See Appendix \ref{sec:choices} for the full cyclotomic extension.
\begin{thm}[{\hspace{1sp}\cite[Theorem 12.4.(1) and (3)]{kato-euler-systems}}] \label{thm:kato_iwasawa_cohomologies}
The following statements hold.
\begin{enumerate}
\item $\mathbb{H}^2(T)$ is a finitely generated torsion module over $\Lambda$.
\item $\mathbb{H}^1(T)$ is free of rank one over $\Lambda$ under Assumption \ref{assu:working_assumptions}.(2).
\end{enumerate}
\end{thm}
We recall the Iwasawa main conjecture without $p$-adic zeta functions \`{a} la Kato.
\begin{conj}[Kato's IMC, {\cite[Conjecture 12.10]{kato-euler-systems}, \cite[Conjecture 6.1]{kurihara-invent}}] \label{conj:kato-main-conjecture}
$$\mathrm{char}_{\Lambda} \left(
\dfrac{ \mathbb{H}^1(T) }{ \Lambda \mathbf{z}_{\mathrm{Kato}} } \right)
 = 
\mathrm{char}_{\Lambda} 
\left( 
\mathbb{H}^2(T) \right)$$
where $\mathbf{z}_{\mathrm{Kato}}$ is Kato's zeta element (Definition \ref{defn:z_kato}).
\end{conj}
\begin{rem} \label{rem:strict-selmer-H2} $ $
\begin{enumerate}
\item 
Following the argument of \cite[$\S$6]{kurihara-invent}, \cite[Theorem 7.1.ii)]{kobayashi-thesis},
if $E(\mathbb{Q}_{\infty, p})[p^\infty]$ is finite, then  $\mathrm{Sel}_{\mathrm{str}}(\mathbb{Q}_\infty, E[p^\infty])^\vee$ and $\mathbb{H}^2(T)$ are pseudo-isomorphic as $\Lambda$-modules where
$\mathrm{Sel}_{\mathrm{str}}(\mathbb{Q}_\infty, E[p^\infty])$ is the $p$-strict (``fine") Selmer group of $E$ over $\mathbb{Q}_{\infty}$.
\item Indeed, due to the argument of \cite[$\S$13.13]{kato-euler-systems},  
if $E$ has potentially good reduction at $p$, then $E(\mathbb{Q}_{\infty, p})[p^\infty]$ is finite.
\item The following statements are equivalent:
\begin{enumerate}
\item $E$ has potentially good reduction at $p$.
\item The corresponding local automorphic representation at $p$ is not (a twist of) Steinberg. (See \cite[$\S$15]{rohrlich-weil-deligne}.)
\item The $j$-invariant of $E$ is $p$-integral. (See \cite[Proposition 10.2.33]{qing-liu}.)
\end{enumerate}
\end{enumerate}
\end{rem}
\subsection{Main Theorems}
\begin{thm}[The main conjecture] \label{thm:main_thm_iwasawa}
Let $E$ be an elliptic curve with additive reduction at $p >7$ satisfying Assumption \ref{assu:working_assumptions}.
If $$\widetilde{\delta}_n \neq 0 \in \mathbb{F}_p$$
for some square-free product of Kolyvagin primes $n$,
then we have
$$\mathrm{char}_{\Lambda} \left(
\dfrac{ \mathbb{H}^1(T) }{ \Lambda \mathbf{z}_{\mathrm{Kato}} } \right)
 = 
\mathrm{char}_{\Lambda} 
\left( 
\mathrm{Sel}_{\mathrm{str}}(\mathbb{Q}_\infty, E[p^\infty])^\vee \right).$$
If we further assume that $E$ has potentially good reduction at $p$, then Kato's IMC (Conjecture \ref{conj:kato-main-conjecture}) holds, i.e.
$$\mathrm{char}_{\Lambda} \left(
\dfrac{ \mathbb{H}^1(T) }{ \Lambda \mathbf{z}_{\mathrm{Kato}} } \right)
 = 
\mathrm{char}_{\Lambda} 
\left( 
\mathbb{H}^2(T) \right) .$$
\end{thm}

\begin{thm}[The $p$-part of BSD formula] \label{thm:main_thm_bsd}
Let $E$ be an elliptic curve with additive reduction at $p >7$ satisfying Assumption \ref{assu:working_assumptions}.
Suppose that $L(E,1) \neq 0$.
If 
$$\widetilde{\delta}_n \neq 0 \in \mathbb{F}_p$$
for some square-free product of Kolyvagin primes $n$,
then the $p$-part of Birch and Swinnerton-Dyer formula for $E$ holds, i.e.
$$\mathrm{ord}_p  \left( \# \textrm{\cyr SH}(E/\mathbb{Q})[p^\infty] \right)
=
\mathrm{ord}_p  \left( \dfrac{L(E,1)}{\Omega^+_E} \right) .$$
\end{thm}
\begin{rem} \label{rem:exceptional_case}$ $
\begin{enumerate}
\item Even in the $p \leq 7$ case, Theorem \ref{thm:main_thm_iwasawa} and Theorem \ref{thm:main_thm_bsd} hold if $(E,p)$ does \textbf{not} satisfy Assumption \ref{assu:exceptional_cases}.
\item It is expected that there always exists 
a square-free product of Kolyvagin primes $n$
such that $\widetilde{\delta}_n \neq 0 \in \mathbb{F}_p$. Practically, it is easy to find such $n$'s. (c.f.\cite{kurihara-iwasawa-2012}.)
\item We do not know whether Theorem \ref{thm:main_thm_iwasawa} directly implies Theorem \ref{thm:main_thm_bsd} or not since there is neither a Mazur-Greenberg style main conjecture nor a control theorem for the additive reduction case.
\item If we replace $L(E,1)$ by the $N$-imprimitive $L$-value $L^{(N)}(E,1)$ in Theorem \ref{thm:main_thm_bsd}, then we can weaken Assumption \ref{assu:working_assumptions}.(1) by $p \nmid \mathrm{Tam}(E)$. (c.f. \cite[Theorem 6.2.4]{mazur-rubin-book}.)
\end{enumerate}
\end{rem}

\section{Computing the integral lattice} \label{sec:computing-lattice}
The goal of this section is to extend \cite[Proposition 3.5.1]{rubin-book} to the additive reduction case over unramified extensions of $\mathbb{Q}_p$. More precisely, we compute the image of the logarithm map of the unramified local points of an elliptic curve with additive reduction.

The main idea is to replace a given elliptic curve with additive reduction by a different elliptic curve with good reduction over a ramified extension of degree 6 such that their generic fibers are isomorphic over the ramified extension following \cite{pannekoek} and \cite{kosters-pannekoek}.
Then we apply the theory of formal groups to the elliptic curve with good reduction. The restriction $p>7$ appears here.
In the $p \leq 7$ case, there are some exceptional cases as described in $\S$\ref{subsec:exceptional_cases}. We refer to \cite{pannekoek} and \cite{kosters-pannekoek} for details.
See also \cite[Lemma 1]{swinnerton-dyer-density}.

\subsection{The integral lattice}
Let $K$ be a finite unramified extension of $\mathbb{Q}_p$ and $k$ be the residue field.
Since $K/\mathbb{Q}_p$ is unramified, the base change to $K$ does not change the reduction type of $E$ ({\hspace{1sp}\cite[Proposition VII.5.4.(a)]{silverman}}).

For a given elliptic curve $E/\mathbb{Q}$ with additive reduction, we consider an elliptic curve  $E/\mathbb{Q}_p$ by taking the base change.
Let $\mathcal{E}$ be a minimal Weierstrass model of $E$ over $\mathbb{Z}_p$ given by a Weierstrass equation
\begin{equation} \label{eqn:weierstrass}
y^2 + a_1 xy + a_3 y = x^3 + a_2 x^2 + a_4 x + a_6
\end{equation}
where $a_i \in p\mathbb{Z}_p$ for each $i$.
By the coordinate change, it is not difficult to see that such a minimal Weierstrass model always exists.
See \cite[Lemma 9]{pannekoek} and \cite[Lemma 9]{kosters-pannekoek}, for example.
Note that $\mathcal{E}(\mathcal{O}_K) = \mathcal{E}(K) = E(K)$ since $\mathcal{E}$ is projective.
Let $\widetilde{E}/\mathbb{F}_p$ be the reduction of $\mathcal{E}$ modulo $p$ and
$\widetilde{E}_{\mathrm{ns}}/\mathbb{F}_p$ be the nonsingular locus of $\widetilde{E}/\mathbb{F}_p$.
Let $E_0(K) \subset E(K)$ be the inverse image of $\widetilde{E}_{\mathrm{ns}}(k)$ and $E_1(K) \subset E(K)$ be the inverse image of the identity of $\widetilde{E}_{\mathrm{ns}}(k)$, i.e. the kernel of the reduction map.
They can also be explicitly written as
\begin{align}
\begin{split} \label{eqn:points_of_reductions}
E_0(K) & = \left\lbrace (x,y) \in E(K) : v_K(x) \leq 0, v_K(y) \leq 0  \right\rbrace \cup \left\lbrace  \infty \right\rbrace , \\
E_1(K) & = \left\lbrace (x,y) \in E(K) : v_K(x) \leq -2, v_K(y) \leq -3  \right\rbrace \cup \left\lbrace  \infty \right\rbrace 
\end{split}
\end{align}
where $v_K$ is the normalized valuation of $K$.
\begin{thm}[{\hspace{1sp}\cite[Theorem 1]{pannekoek}, \cite[Theorem 1, Corollary 2]{kosters-pannekoek}}] \label{thm:computing_integral_lattices}
Let $E/\mathbb{Q}_p$ be an elliptic curve with additive reduction and $K/\mathbb{Q}_p$ be a finite unramified extension.
Then the extension of the formal logarithm map on $E$ induces an isomorphism
$$\mathrm{log}_E : E_0(K) \otimes \mathbb{Z}_p \simeq \mathcal{O}_K. $$
if $(E,p)$ does not satisfies Assumption \ref{assu:exceptional_cases}.
\end{thm}
\begin{prop}
Let $E/\mathbb{Q}_p$ be an elliptic curve with additive reduction at an odd prime $p$ and $K/\mathbb{Q}_p$ be a finite unramified extension.
Then $p$ does not divide $[E(K):E_0(K)] $. 
\end{prop}
\begin{proof}
It is well-known that $[E(K):E_0(K)] \leq 4$ (\hspace{1sp}\cite[Theorem VII.6.1]{silverman}).
Thus, the statement is automatic if $p \geq 5$.
For $p =3$, \cite[Table 4.1, Page 365]{silverman2} shows that
$$p \nmid [E(\mathbb{Q}^{\mathrm{ur}}_p):E_0(\mathbb{Q}^{\mathrm{ur}}_p)] .$$ 
Since the formation of N\'{e}ron models commutes with unramified (\'{e}tale) base change (\hspace{1sp}\cite[Proposition 2.(c), $\S$1.2, Chapter 1]{neron-models}), we have
$$p \nmid [E(K):E_0(K)]$$ 
following the argument in \cite[Proof of Proposition 14]{bradshaw-stein}.
\end{proof}
\begin{cor}
The images of $E(K) \otimes \mathbb{Z}_p$ and $E_0(K)  \otimes \mathbb{Z}_p$ under the logarithm map coincide as $\mathcal{O}_K$.
\end{cor}
Due to \cite[Corollary IV.9.1]{silverman2}, the identity component of the N\'{e}ron model of $E$ over $\mathbb{Z}_p$ is isomorphic to $\mathcal{E}_0$ over $\mathbb{Z}_p$. Thus, we can choose a N\'{e}ron differential $\omega_E$ as a basis of the cotangent space of $\mathcal{E}$ over $\mathcal{O}_K$ (up to a $p$-adic unit). Also, $\omega_E$ corresponds to $f(z)dz$ up to a $p$-adic unit via the modular parametrization under Assumption \ref{assu:working_assumptions}.(3).

Let  $V = T \otimes \mathbb{Q}_p$ and
$\mathbf{D}_{\mathrm{dR}}(V) / \mathbf{D}^0_{\mathrm{dR}}(V)$ be the tangent space of $E$ over $K$ and
$\mathbf{D}^0_{\mathrm{dR}}(V)$ be the cotangent space of $E$ over $K$.
Then we have a perfect pairing
$$\mathrm{Tr} \circ \left\langle -,- \right\rangle_{\mathrm{dR}} : \mathbf{D}^0_{\mathrm{dR}}(V) \times \mathbf{D}_{\mathrm{dR}}(V) / \mathbf{D}^0_{\mathrm{dR}}(V) \to K \to \mathbb{Q}_p$$
where the first map is induced from the local Tate duality and the second map is the trace map.
Let $\omega^*_E \in \mathbf{D}_{\mathrm{dR}}(V) / \mathbf{D}^0_{\mathrm{dR}}(V)$ be the dual basis to $\omega_E$ such that $\left\langle \omega^*_E , \omega_E \right\rangle_{\mathrm{dR}} = 1$.
The map
\[
\xymatrix{
\widehat{E}(\mathcal{O}_K) \ar[r]^-{\mathrm{log}_E}_-{\simeq} & \mathcal{O}_K \ar[r]^-{\cdot \omega^*_E}_-{\simeq} & \mathcal{O}_K \omega^*_E \subseteq \mathbf{D}_{\mathrm{dR}}(V) / \mathbf{D}^0_{\mathrm{dR}}(V) 
}
\]
yields a natural integral structure on $\mathbf{D}^0_{\mathrm{dR}}(V)$.
By the local Tate duality, we have the following statement.
\begin{cor} \label{cor:integral_lattice}
The image of $\mathrm{H}^1(K,T)/\mathrm{H}^1_f(K,T)$ under $\mathrm{exp}^*$ is $\mathcal{O}_K \omega_E$
where $\mathrm{H}^1_f(K,T)$ is the image of the Kummer map of $E(K) \otimes \mathbb{Z}_p$.
\end{cor}

\subsection{Formal groups of elliptic curves} \label{subsec:formal_groups}
Let $\widehat{\mathcal{E}}$ be the formal group of elliptic curve $\mathcal{E}$ over $\mathbb{Z}_p$
and $\mathscr{F}_{\mathcal{E}}(X,Y)$ be the corresponding formal group law (as a two variable power series over $\mathbb{Z}_p$).
Considering $a_i$'s in Equation (\ref{eqn:weierstrass}) as variables, we can regard
$$\mathscr{F}_{\mathcal{E}}(X,Y) \in \mathbb{Z}_p[a_1, a_2, a_3, a_4, a_6]\llbracket X, Y\rrbracket .$$
We also regard a $\mathbb{Z}_p[a_1, a_2, a_3, a_4, a_6]$ as a weighted ring with weight function $\mathrm{wt}(a_i) = i$ for each $i$. Then the coefficients of $\mathscr{F}_{\mathcal{E}}(X,Y)$ in degree $n$ are homogeneous of weight $n-1$ by \cite[Proposition IV.1.1.(c)]{silverman}.
We also have an explicit isomorphism of topological groups
$$\psi_{K'} : E_1(K') = \mathcal{E}_1(K') \simeq \widehat{\mathcal{E}}(\mathfrak{m}_{K'})$$
defined by $(x, y) \mapsto - \frac{x}{y}$ for any finite extension $K'$ of $\mathbb{Q}_p$ (\hspace{1sp}\cite[Proposition VII.2.2]{silverman}).
\subsection{Exceptional cases} \label{subsec:exceptional_cases}
\begin{assu}[Exceptional cases] \label{assu:exceptional_cases}
Following the notation of Equation (\ref{eqn:weierstrass}), if one of the following conditions
\begin{enumerate}
\item $p = 2$ and $a_1 + a_3 \equiv 2 \pmod{4}$,
\item $p = 3$ and $a_2 \equiv 6 \pmod{9}$,
\item $p = 5$ and $a_4 \equiv 10 \pmod{25}$, or
\item $p = 7$ and $a_6 \equiv 14 \pmod{49}$,
\end{enumerate} 
holds, then we call it an \textbf{exceptional case}.
\end{assu}
In the exceptional cases, $E_0(K) = \mathcal{E}_0(K)$ admits non-trivial $p$-torsions and we only have
$$\mathrm{log}_E : \mathcal{E}_0(K) \otimes \mathbb{Z}_p \to \mathfrak{m}_K $$
as described in \cite[Theorem 1]{pannekoek} and \cite[Theorem 1, Corollary 2]{kosters-pannekoek}. We do not consider the exceptional cases in this article.
\subsection{General case: $p>7$}
From now on, we assume $p >7$.
\begin{defn}[Extended convergence of formal groups]
Let $\mathscr{F}_{\mathcal{E}}(X, Y)$ be the formal group law for $\mathcal{E}$ (\hspace{1sp}\cite[Chapter IV]{silverman}).
Then $\mathscr{F}_{\mathcal{E}}(X, Y)$ converges to an element of $\mathcal{O}_K$ for all $X, Y \in \mathcal{O}_K$ since all $a_i \in p\mathbb{Z}_p$.
Thus, we can extend the formal group structure on $\widehat{\mathcal{E}}(\mathfrak{m}_K)$ to $\widehat{\mathcal{E}}(\mathcal{O}_K) $.
\end{defn}
Note that the weighted ring structure of the formal group law in $\S$\ref{subsec:formal_groups} is used in the definition.

\begin{proof}[Proof of Theorem \ref{thm:computing_integral_lattices}]
Let $L := K(\sqrt[6]{p})$ and $\varpi \in L$ be an element of $L$ satisfying $\varpi^6 = p$, and $\mathcal{C}$ be a smooth projective curve over $\mathcal{O}_L$ embedded in $\mathbb{P}^2_{\mathcal{O}_L}$ defined by the following Weierstrass equation
$$y^2 + \frac{a_1}{\varpi^1} xy + \frac{a_3}{\varpi^3} y = x^3 + \frac{a_2}{\varpi^2} x^2 + \frac{a_4}{\varpi^4} x^4 + \frac{a_6}{\varpi^6} .$$
We also define $\mathcal{C}_0$ and $\mathcal{C}_1$ in the same manner. Note that all the coefficients lie in $\mathcal{O}_L$ due to $\varpi^6 = p$. 
Then there exists a birational map $\phi$ over $\mathcal{O}_L$
\[
\xymatrix{
\mathcal{E} \times_{\mathrm{Spec}(\mathcal{O}_K)} \mathrm{Spec}(\mathcal{O}_L) \ar@{-->}[r]^-{\phi} & \mathcal{C}
}
\] 
given by $\phi(x,y) = (\frac{x}{\varpi^2}, \frac{y}{\varpi^3})$ pointwisely. Note that $\phi$ would not be defined at the singular point in the special fiber.
However, $\phi$ becomes an isomorphism on the generic fibers over $L$ and sends the point at infinity of $\mathcal{E}$ to that of $\mathcal{C}$ by definition.
Thus, $\mathcal{E}(L)$ and $\mathcal{C}(L)$ are isomorphic as topological groups.
Furthermore, $\phi$ induces a set-theoretic bijection between $\mathcal{E}_0(L)$ and $\mathcal{C}_1(L)$
due to Equation (\ref{eqn:points_of_reductions}), so it is an isomorphism of topological groups.
We have the following commutative diagram (\emph{a priori} of sets)
\[
\xymatrix{
\mathcal{E}_0(K) \ar@{^{(}->}[r] \ar[d]^-{\Psi_K} & \mathcal{E}_0(L) \ar[r]^-{\phi}_-{\simeq} \ar[d]^-{\Psi_L} & \mathcal{C}_1(L) \ar[d]^-{\psi_L}_-{\simeq} \\
\widehat{\mathcal{E}}(\mathcal{O}_K) \ar@{^{(}->}[r] & \widehat{\mathcal{E}}(\mathcal{O}_L) \ar[r]^-{\varpi \times}_-{\simeq} & \widehat{\mathcal{C}}(\mathfrak{m}_L)
}
\]
where $\Psi_K$ and $\Psi_L$ are defined by the assignment $(x,y) \mapsto - \frac{x}{y}$.

Let $\mathscr{F}_{\mathcal{C}}$ be the formal group law of $\mathcal{C}$ over $\mathcal{O}_L$. By looking at the Weierstrass equations and the corresponding formal group laws, it is easy to see that
$$\varpi \cdot \mathscr{F}_{\mathcal{E}, \mathcal{O}_L}(X,Y) = \mathscr{F}_{\mathcal{C}} (\varpi \cdot X, \varpi \cdot Y) .$$
where $\mathscr{F}_{\mathcal{E}, \mathcal{O}_L} = \mathscr{F}_{\mathcal{E}} \otimes_{\mathbb{Z}_p} \mathcal{O}_L$.

Thus, $\phi$, $\varpi \times $, and $\psi_L$ in the diagram are all isomorphisms of topological groups, so $\Psi_L$ is also an isomorphism of topological groups. 
Thus, $\mathcal{E}_0(K) \simeq \mathrm{Im}(\Psi_K) \subseteq \widehat{\mathcal{E}}(\mathcal{O}_L)$.
Also, we can easily see $\widehat{\mathcal{E}}(\mathcal{O}_K) = \mathrm{Im}(\Psi_K)$ by considering 
 $\Psi_K : (x,y) \mapsto - \frac{x}{y}$.
Thus, $\Psi_K$ is also an isomorphism of topological groups.

Since we assume $p > 7$, we have $v_L(p) =  6 v_K(p) = 6 e(K/\mathbb{Q}_p) <  p-1$.
Thus, the formal logarithm on $\mathcal{C}$ induces an isomorphism
$$\mathrm{log}_{\widehat{\mathcal{C}}} : \widehat{\mathcal{C}}(\mathfrak{m}_L) \simeq \mathfrak{m}_L .$$
Since $\mathcal{E}_0(K)$ embeds into $\widehat{\mathcal{C}}(\mathfrak{m}_L)$,
$\mathcal{E}_0(K)$ is also torsion-free.
Since the formal logarithm map extends to $\mathrm{log} \otimes \mathbb{Q}_p : \mathcal{E}(K) \otimes \mathbb{Q}_p \simeq K$, we have
\[
\xymatrix{
0 \ar[r] & \mathcal{E}_1(K) \ar[r] \ar[d]^-{\psi_K}_{\simeq} &  \mathcal{E}_0(K) \ar[r] \ar[d]^-{\Psi_K}_{\simeq} &  k \ar[r] \ar@{=}[d] & 0 \\
0 \ar[r] & \widehat{\mathcal{E}}(\mathfrak{m}_K) \ar[r] \ar[d]^-{\mathrm{log}_{\widehat{\mathcal{E}}}}_{\simeq} & \widehat{\mathcal{E}}(\mathcal{O}_K) \ar[r] \ar[d]^-{\mathrm{log}_{E}}_{\simeq} &  k \ar[r] \ar@{=}[d] & 0 \\
0 \ar[r] & \mathfrak{m}_K \ar[r] & \mathcal{O}_K \ar[r] &  k \ar[r] & 0 
}
\]
and it proves Theorem \ref{thm:computing_integral_lattices}.
\end{proof}

\section{The main conjecture: review of {\cite{kks}}}
In this section, we present a slightly different computation from that of \cite{kks} due to a different normalization of Kato's Euler system.
Let $E$ be an elliptic curve over $\mathbb{Q}$ with additive reduction at $p$; especially, the Euler factor at $p$ of $L$-function $L(E,s)$ is 1.
\subsection{Kato's Euler systems and modular symbols} \label{subsec:euler_systems_modular_symbols}
Let $c_{\mathbb{Q}(\mu_n)} (1) \in \mathrm{H}^1(\mathbb{Q}(\mu_n), T)$ be Kato's Euler system for $T$ at $\mathbb{Q}(\mu_n)$ such that $$\mathrm{exp}^*(c_{\mathbb{Q}}(1)) = \frac{L^{(Np)}(E,1)}{\Omega^+_E} $$
where $L^{(Np)}(E, \chi, 1)$ is the $Np$-imprimitive $L$-value of $E$ at $s = 1$ twisted by $\chi$.
See Definition \ref{defn:the_defn_ES} for the precise choice of $c_{\mathbb{Q}(n)} (1)$.
Following \cite[Theorem 3.5.1]{rubin-book} and \cite[Theorem 6.6 and Theorem 9.7]{kato-euler-systems}, we have the following zeta value formula.
\begin{thm} \label{thm:kato_interpolation}
Let $E$ be an elliptic curve over $\mathbb{Q}$ with additive reduction at $p$.
Let $\chi$ be a Dirichlet character of $p$-power exponent and conductor $n$ with $(n,Np)=1$.
Then we have 
\begin{align} 
\begin{split}
\label{eqn:algebraic_formula}
& \sum_{c \in (\mathbb{Z}/n\mathbb{Z})^\times}  \chi(c) \cdot \left\langle \omega^*_{E},  \mathrm{exp}^{*} \left(  \left( c^{\chi(-1)}_{\mathbb{Q}(\mu_n)}(1) \right)^{\sigma_c} \right) \right\rangle_{\mathrm{dR}} \\
= &  \left( \prod_{q \mid N_{\mathrm{sp}}} ( 1-  q^{-1} \chi(q) )  \right) \cdot \left( \prod_{q \mid N_{\mathrm{ns}}} ( 1 + q^{-1} \chi(q) )  \right) \cdot \frac{L^{(p)}(E, \chi, 1)}{\Omega^{\chi(-1)}_{E}} 
\end{split}
\end{align}
where $c^{\chi(-1)}_{\mathbb{Q}(\mu_n)}(1)$ lies in the eigenspace with respect to the complex conjugation with eigenvalue $\chi(-1)$ and $L^{(p)}(E, \chi, 1)$ is the $p$-imprimitive $L$-value of $E$ at $s = 1$ twisted by $\chi$.
\end{thm}
\begin{rem}
See $\S$\ref{subsec:extension_of_euler_systems} for the interpolation formula for such Dirichlet characters.
\end{rem}
We rewrite the last term in Equation (\ref{eqn:algebraic_formula}) in terms of modular symbols. 
Following \cite[(2.2)]{pollack-oms}, we have
\begin{equation} \label{eqn:L-value}
L(E, \chi, 1) = \frac{\tau(\chi)}{n} \cdot \sum_{a \in (\mathbb{Z}/n\mathbb{Z})^\times} \overline{\chi}(a) \cdot 2\pi i \cdot \int^{-a/n}_{i\infty} f(z)dz 
\end{equation}
where $\tau(\chi)$ is the Gauss sum of $\chi$. 
Let $G_n = (\mathbb{Z}/n\mathbb{Z})^\times$, $G_{n,p}$ be the $p$-part of $G_n$, and $G^p_{n}$ be the prime-to-$p$ part of $G_n$. 
Expanding the Gauss sum in Equation (\ref{eqn:L-value}), we have
\begin{align*}
\chi(-1) \cdot \frac{L(E,\chi, 1)}{\Omega^{\chi(-1)}_{E}} &
 =  \frac{1}{n}  \cdot   \sum_{c \in (\mathbb{Z}/n\mathbb{Z})^\times} \chi(c) \cdot \sigma_c \cdot  \left( \sum_{a \in (\mathbb{Z}/n\mathbb{Z})^\times} \zeta^{a}_n \cdot  \left[ \frac{-a}{n} \right]^{\chi(-1)} \right)  \\
& =  \frac{1}{n}  \cdot   \sum_{c_1 \in G_{n,p}} \sum_{c_2 \in G^{p}_{n}} \chi(c_1 c_2) \cdot \sigma_{c_1} \cdot \sigma_{c_2} \cdot  \left( \sum_{a \in (\mathbb{Z}/n\mathbb{Z})^\times} \zeta^{a}_n \cdot  \left[ \frac{-a}{n} \right]^{\chi(-1)} \right) \\
& =  \frac{1}{n}  \cdot   \sum_{c_1 \in G_{n,p}} \sum_{c_2 \in G^{p}_{n}} \chi(c_1) \cdot \sigma_{c_1} \cdot \sigma_{c_2} \cdot  \left( \sum_{a \in (\mathbb{Z}/n\mathbb{Z})^\times} \zeta^{a}_n \cdot  \left[ \frac{-a}{n} \right]^{\chi(-1)} \right) .
\end{align*}
Let $\mathbb{Q}(n)$ be the maximal $p$-subextension of $\mathbb{Q}$ in $\mathbb{Q}(\mu_n)$ and $\mathcal{O}_{\mathbb{Q}(n)}$ the ring of integers of $\mathbb{Q}(n)$.
We define the values
$$c^{\mathrm{an}, \pm}_{\mathbb{Q}(n)} :=  \frac{\pm 1}{n} \cdot \left( \prod_{q \mid N_{\mathrm{sp}}} ( 1-  q^{-1} \sigma^{-1}_q )  \right) \cdot  \left( \prod_{q \mid N_{\mathrm{ns}}} ( 1 + q^{-1} \sigma^{-1}_q )  \right) \cdot  \sum_{c_2 \in G^{p}_{n}}  \sigma_{c_2} \cdot \left( \sum_{a \in (\mathbb{Z}/n\mathbb{Z})^\times} \zeta^{a}_n \cdot \left[ \frac{-a}{n} \right]^{\pm}_{f} \right) \in \mathbb{Z}_p \otimes \mathcal{O}_{\mathbb{Q}(n)}$$
in order to have
\begin{equation} \label{eqn:analytic_formula}
 \sum_{c_1 \in G_{n,p}} \left(   \sigma_{c_1} \left(  c^{\mathrm{an}, \chi(-1)}_{\mathbb{Q}(n)}  \right) \right) \cdot \chi(c_1) = \frac{L^{(Np)}(E, \chi, 1)}{\Omega^{\chi(-1)}_{E}} .
\end{equation}
Indeed, $c^{\mathrm{an}, -}_{\mathbb{Q}(n)} = 0$ since $\mathbb{Q}(n)$ is totally real.

Comparing the coefficients of (\ref{eqn:algebraic_formula}) and (\ref{eqn:analytic_formula}), we can easily observe the following statement.
\begin{prop} \label{prop:comparing_coeff}
For any $c_1 \in G_{n,p}$, we have
$$\left\langle \omega^*_{E},  \mathrm{exp}^{*} \left(  \left( c^{\chi(-1)}_{\mathbb{Q}(n)}(1) \right)^{\sigma_{c_1}} \right) \right\rangle_{\mathrm{dR}} =  \sigma_{c_1} \left(  c^{\mathrm{an}, \chi(-1)}_{\mathbb{Q}(n)}  \right) $$
where $c^{\chi(-1)}_{\mathbb{Q}(n)}(1)$ is Kato's Euler system at $\mathbb{Q}(n)$.
\end{prop}

\subsection{Kolyvagin derivatives on modular symbols and Kurihara numbers} \label{subsec:williams}
Let $\ell$ be a Kolyvagin prime and $\eta_\ell$ be a primitive root modulo $\ell$.
Let $$D_\ell := \sum_{i = 0}^{\ell-2} i \sigma^{i}_{\eta_\ell} \in \mathbb{Z}[\mathrm{Gal}(\mathbb{Q}(\mu_\ell)/\mathbb{Q})]$$ be the Kolyvagin derivative operator for $\mathrm{Gal}(\mathbb{Q}(\mu_\ell)/\mathbb{Q})$ with respect to $\eta_\ell$ and $D_n = \prod_{\ell \vert n} D_\ell$.
Let $p^k \Vert (\ell-1)$. Then we let
$$\overline{D}_\ell := \sum_{i = 0}^{p^k-1} i \sigma^{\frac{\ell-1}{p^k} i}_{\eta_\ell} \in  \mathbb{Z}[\mathrm{Gal}(\mathbb{Q}(\ell)/\mathbb{Q}) ]$$
also be the Kolyvagin derivative operator for $\mathrm{Gal}(\mathbb{Q}(\ell)/\mathbb{Q})$ with respect to $\eta^{\frac{\ell-1}{p^k}}_{\ell}$ and $\overline{D}_n = \prod_{\ell \vert n} \overline{D}_\ell$.
Let $a, b \in \mathbb{Z}$ such that  
$$\frac{\ell-1}{p^k} \cdot a + p^k \cdot b = 1 .$$
By rewriting
$$D_\ell = \sum_{i = 0}^{\ell-2}  \left( \frac{\ell-1}{p^k} a + p^k b \right) \cdot i \cdot \sigma^{  ( \frac{\ell-1}{p^k} a + p^k b ) i }_{\eta_\ell} ,$$
the natural quotient map
$$\mathbb{Z}[\mathrm{Gal}(\mathbb{Q}(\mu_\ell)/\mathbb{Q})] \to \mathbb{Z}[\mathrm{Gal}(\mathbb{Q}(\ell)/\mathbb{Q}) ]$$
yields
$$D_\ell \mapsto
 \left( \frac{\ell-1}{p^k}   \right)
 \cdot
\left( \sum_{i = 0}^{\ell-2} a \cdot i \cdot \sigma^{ \frac{\ell-1}{p^k} a i }_{\eta_\ell} \right)
+
 p^k \cdot \left(  \sum_{i = 0}^{\ell-2}  b \cdot i \cdot \sigma^{ \frac{\ell-1}{p^k} a i  }_{\eta_\ell} \right) 
$$
since $\sigma^{p^k}_{\eta_{\ell}}$ maps to 1.
We also have
\begin{equation} \label{eqn:comparing_derivatives}
\sum_{i = 0}^{\ell-2}  a \cdot i \cdot \sigma^{ \frac{\ell-1}{p^k} a i }_{\eta_\ell}
\equiv 
 \left( \frac{\ell-1}{p^k}   \right)
 \cdot \sum_{i = 0}^{p^k-1}  a \cdot  i \cdot \sigma^{ \frac{\ell-1}{p^k} a i }_{\eta_\ell} \Mod{p} 
\end{equation}
in $\mathbb{F}_p[\mathrm{Gal}(\mathbb{Q}(\ell)/\mathbb{Q}) ]$. Since $a$ is non-zero mod $p$, we can replace $a \cdot i$ by $i$ in the RHS of (\ref{eqn:comparing_derivatives}).

Repeating this argument for each Kolyvagin prime dividing $n$, we have the following lemma, which compares Kolyvagin derivatives for 
$\mathrm{Gal}(\mathbb{Q}(\mu_n)/\mathbb{Q})$ and $\mathrm{Gal}(\mathbb{Q}(n)/\mathbb{Q})$.
\begin{lem} \label{lem:comparing_kolyvagin_derivatives}
$$D_n \equiv u \cdot \overline{D}_n$$ 
for some $u \in \mathbb{F}^\times_p$
in $\mathbb{F}_p[\mathrm{Gal}(\mathbb{Q}(n)/\mathbb{Q}) ]$.
\end{lem}

Considering the mod $p$ Taylor expansion of Mazur-Tate elements at ${\displaystyle \prod^{s}_{i=1} \left( \sigma_{\eta_{\ell_i}} - 1\right) }$ where $n = \prod_{i=1}^s \ell_i$, we observe the following statement.
\begin{prop}[{\hspace{1sp}\cite[Theorem 7.5]{kks}}] \label{prop:computation_KS}
Let $n$ be a square-free product of Kolyvagin primes and $D_n$ be ``the" Kolyvagin derivative operator in $\mathbb{Z}_p[\mathrm{Gal}(\mathbb{Q}(\mu_n)/\mathbb{Q})]$.
We have the following equalities in $\mathbb{F}_p$
\begin{align*}
D_n \left( \sum_{a \in (\mathbb{Z}/n\mathbb{Z})^\times } \zeta^{ a'}_n \left[ \frac{a}{n}\right]^{\pm}_f \right) & \equiv
\sum_{a \in (\mathbb{Z}/n\mathbb{Z})^\times}  \left( \prod_{\ell \vert n}  \mathrm{log}_{\mathbb{F}_\ell} ( a ) \right) \cdot \left[\frac{a}{n}  \right]^{\pm}_f  \Mod{p} 
\end{align*}
where $a' = \pm a$.
\end{prop}
\begin{rem}
In Proposition \ref{prop:computation_KS}, the primitive roots modulo primes dividing $n$ for $D_n$ and $\prod_{\ell \vert n} \mathrm{log}_{\mathbb{F}_\ell}$ are chosen to match up. Thus, we can say ``the" Kolyvagin derivative. 
The congruence still works up to multiplication by an element in $\mathbb{F}^\times_p$ even if we do not match up the choices.
\end{rem}

\subsection{Proof of Theorem {\ref{thm:main_thm_iwasawa}}} \label{subsec:the_proof}
We give a proof of Theorem \ref{thm:main_thm_iwasawa}.
By \cite[Theorem 5.3.10.(iii)]{mazur-rubin-book}, it suffices check two conditions:
\begin{enumerate}
\item the non-triviality of $\kappa^{\infty}_1 \in \mathrm{H}^1(\mathbb{Q}_\infty, T)$.
\item the $\Lambda$-primitivity of the $\Lambda$-adic Kolyvagin system $\ks^\infty$.
\end{enumerate}
Th non-triviality of $\kappa^{\infty}_1$ comes from the generic non-vanishing of cyclotomic twists of $L$-values (even when $p$ divides $N$) following \cite{rohrlich-nonvanishing-2} and the dual exponential map. Note that there is no relevant Coleman map in the additive reduction case yet.

Also, it is not difficult to observe that the $\Lambda$-primitivity of the $\Lambda$-adic Kolyvagin system $\ks^\infty$ follows from the primitivity of the Kolyvagin system $\ks$. See \cite[Proposition 4.20]{kks} for detail.

Since $\mathbb{Q}(n)$ is totally real, we have $c_{\mathbb{Q}(n)}(1) = c^+_{\mathbb{Q}(n)}(1)$.
Consider the following commutative diagram
\begin{equation} \label{eqn:derivative_delta_n}
\begin{split}
{ \scriptsize
\xymatrix@C=0.2em@R=1.5em{
\mathrm{H}^1(\mathbb{Q}(n), T) \ar[r]^-{\overline{D}_n} & \mathrm{H}^1(\mathbb{Q}(n), T) \ar[d]^-{\bmod{p}} \ar[rrrr]^-{\left\langle \omega^*_E,  \mathrm{exp}^{*} \left( \mathrm{loc}_p - \right) \right\rangle_{\mathrm{dR}}} & & & & \mathbb{Z}_p \otimes \mathcal{O}_{\mathbb{Q}(n)} \ar[d]^-{\bmod{p}} \\
& \left( \frac{ \mathrm{H}^1(\mathbb{Q}(n), T) }{ p \mathrm{H}^1(\mathbb{Q}(n), T) } \right)^{\mathrm{Gal}( \mathbb{Q}(n)/\mathbb{Q} )} \ar[rrrr]^-{\overline{\left\langle \omega^*_E,  \mathrm{exp}^{*} \left( \mathrm{loc}_p - \right) \right\rangle_{\mathrm{dR}}}} & & & & \mathbb{F}_p \otimes \mathcal{O}_{\mathbb{Q}(n)} \\
 c^+_{\mathbb{Q}(n)}(1) \ar@{|->}[r] &
  \overline{D}_n c^+_{\mathbb{Q}(n)}(1) \ar@{|->}[d] \ar@{|->}[rrrr] & & & & \langle \omega^*_E , \mathrm{exp}^{*}( \mathrm{loc}_p \overline{D}_n  c^{+}_{\mathbb{Q}(n)}(1)) \rangle_{\mathrm{dR}} \ar@{|->}[d]\\
&  d^+_n \ar@{|->}[rrrr] & & & & \langle \omega^*_E , \mathrm{exp}^{*}(\mathrm{loc}_p \overline{D}_n  c^{+}_{\mathbb{Q}(n)}(1)) \rangle_{\mathrm{dR}} \Mod{p} .
}
}
\end{split}
\end{equation}
where
 $\overline{\left\langle \omega^*_E ,  \mathrm{exp}^{*} \left(  - \right) \right\rangle_{\mathrm{dR}}}$ is the induced reduction of $\left\langle \omega^*_E ,  \mathrm{exp}^{*} \left(  - \right) \right\rangle_{\mathrm{dR}}$ modulo $p$.
Because $\overline{D}_n \in \mathbb{Z}[\mathrm{Gal}(\mathbb{Q}(n)/\mathbb{Q})]$ commutes with $\mathrm{loc}_p$, $\mathrm{exp}^*$, and $\langle \omega^*_E , -   \rangle_{\mathrm{dR}}$, we have
$$\langle \omega^*_E ,\mathrm{exp}^{*}(\mathrm{loc}_p \overline{D}_n c^{+}_{\mathbb{Q}(n)}(1)) \rangle_{\mathrm{dR}} = \overline{D}_n \langle \omega^*_E ,\mathrm{exp}^{*}(\mathrm{loc}_p c^{+}_{\mathbb{Q}(n)}(1)) \rangle_{\mathrm{dR}} .$$

\begin{thm} \label{thm:nonvanishing_delta}
If $\kappa_n = 0 \Mod{p}$, then $\widetilde{\delta}_n =0$. 
\end{thm}
The proof is identical with that given in \cite{kks}.
\begin{proof}
Suppose that $\kappa_n \Mod{p} = 0$ in $\mathrm{H}^1_{\mathcal{F}(n)}(\mathbb{Q}, T/ p T) \otimes G_n$.
Then, following \cite[Appendix A]{mazur-rubin-book}, a straightforward computation yields
$$d^+_n = 0 \in \mathrm{H}^1(\mathbb{Q}(n), T)/p \mathrm{H}^1(\mathbb{Q}(n), T).$$
Thus, $\overline{D}_n c^+_{\mathbb{Q}(n)}(1) \in p\mathrm{H}^1(\mathbb{Q}(n), T)$.
By Proposition \ref{prop:comparing_coeff}, if $\overline{D}_n \left\langle \omega^*_E ,  \mathrm{exp}^{*} \left( \mathrm{loc}_p c^{+}_{\mathbb{Q}(n)}(1) \right) \right\rangle_{\mathrm{dR}} \in p\mathbb{Z}_p \otimes \mathcal{O}_{\mathbb{Q}(n)}$, then we have 
\begin{equation} \label{eqn:valuation_derivatives}
\left( \overline{D}_n c^{\mathrm{an}, +}_{\mathbb{Q}(n)} \right) \in p\mathbb{Z}_p \otimes \mathcal{O}_{\mathbb{Q}(n)}.
\end{equation}
Due to Lemma \ref{lem:comparing_kolyvagin_derivatives} and Proposition \ref{prop:computation_KS}, (\ref{eqn:valuation_derivatives}) is equivalent to
$$ \left( \prod_{q \mid N_{\mathrm{sp}}} ( 1-  q^{-1} \sigma^{-1}_q )  \right) \cdot \left( \prod_{q \mid N_{\mathrm{ns}}} ( 1 + q^{-1} \sigma^{-1}_q )  \right)  \cdot  \widetilde{\delta}_n  = 0$$
in $\mathbb{F}_p \otimes \mathcal{O}_{\mathbb{Q}(n)}$ and indeed, $\widetilde{\delta}_n \in \mathbb{F}_p$; thus, it is equivalent to
$$\left( \prod_{q \mid N_{\mathrm{sp}}} ( 1-  q^{-1} )  \right) \cdot \left( \prod_{q \mid N_{\mathrm{ns}}} ( 1 + q^{-1}  )  \right) \cdot  \widetilde{\delta}_n = 0 \in  \mathbb{F}_{p} .$$
Due to Assumption \ref{assu:working_assumptions}.(1), it is equivalent to
$$ \widetilde{\delta}_n = 0 \in  \mathbb{F}_{p} .$$
\end{proof}
As a result of Theorem \ref{thm:nonvanishing_delta}, we obtain an analogue of Conjecture \ref{conj:kato-main-conjecture} with $\mathbf{z}^{(N)}_{\mathrm{Kato}}$ (defined in (\ref{eqn:imprimitive_z_kato})). By Lemma \ref{lem:comparing_main_conj_imprimitive}, Conjecture \ref{conj:kato-main-conjecture} follows.

\section{The $p$-part of the Birch and Swinnerton-Dyer formula for the rank zero case}
The goal of this section is to prove Theorem \ref{thm:main_thm_bsd}, which generalizes \cite[Theorem 6.2.4]{mazur-rubin-book} to the additive reduction case when the corresponding Kolyvagin system is primitive.
\begin{assu}[$\mathrm{Hyp}(\mathbb{Q}, T)$, {\cite[$\S$2.1]{rubin-book}}] $ $
\begin{enumerate}
\item There exists an element $\tau \in G_{\mathbb{Q}}$ such that
\begin{itemize}
\item $\tau$ acts trivially on $\mu_{p^\infty}$, and
\item $T/(\tau - 1)T$ is free of rank one over $\mathbb{Z}_p$.
\end{itemize}
\item $\overline{\rho}$ is irreducible.
\end{enumerate}
\end{assu}

We recall the ``error term" in the Euler system argument.
Let $W = E[p^\infty]$ and $\mathbb{Q}(W)$ be the smallest extension of $\mathbb{Q}$ where $G_{\mathbb{Q}(W)}$ acts trivially on $E[p^\infty]$.
Let
\begin{align*}
\mathfrak{n}_{E[p^\infty]} & :=\mathrm{length}_{\mathbb{Z}_p} \left( \mathrm{H}^1(\mathbb{Q}(W)/\mathbb{Q}, E[p^\infty]) \cap \mathrm{Sel}_{\mathrm{rel}}(\mathbb{Q}, E[p^\infty]) \right), \\
\mathfrak{n}^*_{E[p^\infty]} & :=\mathrm{length}_{\mathbb{Z}_p} \left( \mathrm{H}^1(\mathbb{Q}(W)/\mathbb{Q}, E[p^\infty]) \cap \mathrm{Sel}_{\mathrm{str}}(\mathbb{Q}, E[p^\infty]) \right) .
\end{align*}
By \cite[Proposition 3.5.8.(ii)]{rubin-book}, if the $p$-adic representation $\rho : G_{\mathbb{Q}} \to \mathrm{Aut}_{\mathbb{Z}_p}(T)$ is surjective, then $$\mathrm{H}^1(\mathbb{Q}(W)/\mathbb{Q}, E[p^\infty]) \simeq \mathrm{H}^1(\mathrm{GL}_2(\mathbb{Z}_p), (\mathbb{Q}_p/\mathbb{Z}_p)^2 ) = 0,$$ so
$\mathfrak{n}_{E[p^\infty]} = \mathfrak{n}^*_{E[p^\infty]} = 0$.
Furthermore, if $\overline{\rho}$ is surjective with $p>3$, then $\rho$ is surjective since the determinant of $\rho$ is the cyclotomic character by \cite[Lemma 1, $\S$1]{swinnerton-dyer-congruences}.

The machinery of Euler and Kolyvagin systems yields the following theorem.
\begin{thm} \label{thm:bounding_strict_selmer}
Suppose that $p >2$ and $T$ satisfies $\mathrm{Hyp}(\mathbb{Q},T)$.
Let $\es(1)$ be an Euler system for $T$ and $\ks$ be the corresponding Kolyvagin system with $c_{\mathbb{Q}}(1) = \kappa_1$.
Assume that $\mathrm{loc}^s_p(c_{\mathbb{Q}}(1)) = \mathrm{loc}^s_p(\kappa_1) \neq 0$.
\begin{enumerate}
\item $\mathrm{length}_{\mathbb{Z}_p} \left( \mathrm{Sel}_{\mathrm{str}}(\mathbb{Q},E[p^\infty]) \right)$ is finite if and only if $\kappa_1 \neq 0$.
\item $\mathrm{length}_{\mathbb{Z}_p} \left( \mathrm{Sel}_{\mathrm{str}}(\mathbb{Q},E[p^\infty])  \right) \leq \partial^{(0)} (\ks ) := \mathrm{max}\lbrace  j : \kappa_1 \in p^j \mathrm{Sel}_{\mathrm{rel}}(\mathbb{Q}, T) \rbrace$.
\item If $\ks$ is primitive, then
$$\mathrm{length}_{\mathbb{Z}_p} \left( \mathrm{Sel}_{\mathrm{str}}(\mathbb{Q},E[p^\infty])  \right) = \partial^{(0)} (\ks ) .$$
\end{enumerate}
\end{thm}
\begin{proof}
See \cite[Theorem 2.2.2]{rubin-book} and \cite[Corollary 5.2.13]{mazur-rubin-book} for detail.
\end{proof}


\begin{lem}[{\hspace{1sp}\cite[1.9. Proposition]{coates-sujatha-galois}}] \label{lem:vanishing_H2}
If $\mathrm{Sel}(\mathbb{Q},E[p^\infty])$ is finite, then we have the following statements:
\begin{enumerate}
\item $\mathrm{H}^2 (\mathbb{Q}_{\Sigma}/\mathbb{Q}, E[p^\infty]) = 0$;
\item if we further assume $E(\mathbb{Q})[p]$ is trivial, then the global-to-local map defining Selmer groups is surjective. 
\end{enumerate}
\end{lem}
Indeed, the vanishing of $\mathrm{H}^2(\mathbb{Q}_{\Sigma}/\mathbb{Q}, E[p^\infty])$ is equivalent to the finiteness of the $p$-strict Selmer group. See \cite[Lemma 3.2]{hachimori_fine}.

Using (the Pontryagin dual of) the second sequence of \cite[A.3.2]{perrin-riou-book} with the $p$-strict local condition for $B$ and the $p$-relaxed local condition for $A$, we have the following exact sequence (see also \cite[(7.18)]{kobayashi-thesis})
\begin{equation} \label{eqn:poitou-tate}
{\small
\xymatrix@C=0.7em{
\mathrm{H}^2 (\mathbb{Q}_{\Sigma}/\mathbb{Q}, E[p^\infty])^\vee  \ar[r] & \mathrm{Sel}_{\mathrm{rel}} (\mathbb{Q}, T)
\ar[r] & \dfrac{\mathrm{H}^1 (\mathbb{Q}_p, T)}{E(\mathbb{Q}_p) \otimes \mathbb{Z}_p}
\ar[r] &  \mathrm{Sel} (\mathbb{Q}, E[p^\infty])^\vee  \ar[r] & \mathrm{Sel}_{\mathrm{str}} (\mathbb{Q}, E[p^\infty])^\vee \ar[r] &0 .
}
}
\end{equation}
Note that here we use Lemma \ref{lem:vanishing_H2}.(2) to have the third and the fourth terms.

Furthermore, under the finiteness of the Selmer group (due to Lemma \ref{lem:vanishing_H2}.(1)), (\ref{eqn:poitou-tate}) becomes
\[
{\small
\xymatrix@C=1.0em{
0  \ar[r] & \mathrm{Sel}_{\mathrm{rel}} (\mathbb{Q}, T)
\ar[r] & \dfrac{\mathrm{H}^1 (\mathbb{Q}_p, T)}{E(\mathbb{Q}_p) \otimes \mathbb{Z}_p}
\ar[r] &  \mathrm{Sel} (\mathbb{Q}, E[p^\infty])^\vee  \ar[r] & \mathrm{Sel}_{\mathrm{str}} (\mathbb{Q}, E[p^\infty])^\vee \ar[r] &0 .
}
}
\]
Taking the quotient by the $\mathbb{Z}_p$-module generated by $c_{\mathbb{Q}}(1)$, Kato's Euler system at $\mathbb{Q}$, we have
\begin{equation} \label{eqn:poitou-tate-2}
{\small
\xymatrix@C=1.0em{
0  \ar[r] & \dfrac{ \mathrm{Sel}_{\mathrm{rel}} (\mathbb{Q}, T) } { \mathbb{Z}_p c_{\mathbb{Q}}(1) }
\ar[r] & \dfrac{\mathrm{H}^1_s (\mathbb{Q}_p, T)}{ \mathbb{Z}_p \mathrm{loc}^s_p c_{\mathbb{Q}}(1) }
\ar[r] &  \mathrm{Sel} (\mathbb{Q}, E[p^\infty])^\vee  \ar[r] & \mathrm{Sel}_{\mathrm{str}} (\mathbb{Q}, E[p^\infty])^\vee \ar[r] &0 
}
}
\end{equation}
where $\mathrm{H}^1_s (\mathbb{Q}_p, T) =  \frac{\mathrm{H}^1 (\mathbb{Q}_p, T)}{E(\mathbb{Q}_p) \otimes \mathbb{Z}_p}$.

The following theorem is the generalized and sharpened version of \cite[Theorem 2.2.10.(ii)]{rubin-book}.
\begin{thm} \label{thm:bounding_selmer}
Suppose that $p >2$ and $T$ satisfies $\mathrm{Hyp}(\mathbb{Q},T)$.
Let $\es(1)$ be an Euler system for $T$ and $\ks$ be the corresponding Kolyvagin system with $c_{\mathbb{Q}}(1) = \kappa_1$.
Assume that $\mathrm{loc}^s_p(\kappa_1) \neq 0$ and $\ks$ is primitive.
Then $\mathrm{Sel}(\mathbb{Q}, E[p^\infty])$ is finite and
$$\mathrm{length}_{\mathbb{Z}_p} \left( \mathrm{Sel}(\mathbb{Q}, E[p^\infty]) \right) = \mathrm{length}_{\mathbb{Z}_p} \mathrm{H}^1_s(\mathbb{Q}_p, T) / \mathrm{loc}^s_p(\kappa_1) . $$
\end{thm}
\begin{proof}
The finiteness of Selmer groups is already proved in \cite[Theorem 2.2.10.(ii)]{rubin-book}.
We have
\begin{align*}
\partial^{(0)} (\ks ) & = \mathrm{length}_{\mathbb{Z}_p}  \left( \mathrm{loc}^s_p \left( \mathrm{Sel}_{\mathrm{rel}}(\mathbb{Q}, T) \right) / \mathbb{Z}_p \mathrm{loc}^s_p(\kappa_1) \right) & \textrm{(\ref{eqn:poitou-tate-2})} \\
& < \infty. & ( \mathrm{loc}^s_p (\kappa_1) \neq 0 )
\end{align*}
By Theorem \ref{thm:bounding_strict_selmer}.(3), we have
$$\mathrm{length}_{\mathbb{Z}_p} \left( \mathrm{Sel}_{\mathrm{str}}(\mathbb{Q},E[p^\infty])  \right) = \mathrm{length}_{\mathbb{Z}_p}  \left( \mathrm{loc}^s_p \left( \mathrm{Sel}_{\mathrm{rel}}(\mathbb{Q}, T) \right) / \mathbb{Z}_p \mathrm{loc}^s_p(\kappa_1) \right)  < \infty.$$
Then, by Sequence (\ref{eqn:poitou-tate-2}) again,
$$\mathrm{length}_{\mathbb{Z}_p} \left( \mathrm{Sel}(\mathbb{Q},E[p^\infty])  \right) = \mathrm{length}_{\mathbb{Z}_p}  \left( \mathrm{H}^1_s(\mathbb{Q}_p, T) / \mathbb{Z}_p \mathrm{loc}^s_p(\kappa_1) \right)  < \infty.$$
\end{proof}
Applying Theorem \ref{thm:kato_interpolation}, we extend \cite[Theorem 3.5.11]{rubin-book} to the additive reduction case.
\begin{thm}
Let $E$ be a non-CM elliptic curve over $\mathbb{Q}$ with additive reduction at an odd prime $p > 7$.
Assume that
 $L(E,1) \neq 0$,  $\overline{\rho}$ is surjective, $p \nmid \mathrm{Tam}(E)$ and $\ks$ is primitive.
Then $$ \mathrm{ord}_p \left( \# \sha (E/\mathbb{Q})[p^\infty] \right) =\mathrm{ord}_p \left( \dfrac{L(E,1)}{\Omega^+_E} \right) .$$
\end{thm}
\begin{proof}
By Corollary \ref{cor:integral_lattice}, we have
\[
\xymatrix{
\left\langle \omega^*_E ,  \mathrm{exp}^* \left( \mathrm{H}^1_s(\mathbb{Q}_p, T) \right) \right\rangle = \mathbb{Z}_p \subseteq \mathbb{Q}_p , & \left\langle \omega^*_E ,  \mathrm{exp}^* \left( \mathbb{Z}_p \mathrm{loc}^s_p \kappa_1 \right) \right\rangle = \frac{L(E,1)}{\Omega^+_E}\mathbb{Z}_p .
}
\]
Since $L(E,1) \neq 0$, we have $\mathrm{loc}^s_p(\kappa_1) \neq 0$ via Theorem \ref{thm:kato_interpolation}.
Also, the finiteness of $\mathrm{Sel}(\mathbb{Q}, E[p^\infty])$ and the surjectivity of $\overline{\rho}$ show that
$$\#\mathrm{Sel}(\mathbb{Q}, E[p^\infty]) = \#\sha(E/\mathbb{Q})[p^\infty] .$$
By Theorem \ref{thm:bounding_selmer}, we obtain
$$\mathrm{length}_{\mathbb{Z}_p} \left( \mathrm{Sel}(\mathbb{Q}, E[p^\infty]) \right) = \mathrm{length}_{\mathbb{Z}_p} \mathrm{H}^1_s(\mathbb{Q}_p, T) / \mathrm{loc}^s_p(\kappa_1) . $$
Thus, the conclusion follows.
\end{proof}

\section{Examples}
In this section, we provide three new examples of the main conjecture and the $p$-part of the BSD formula.
The Sage code for an effective computation of Kurihara numbers due to Alexandru Ghitza is available at
\begin{center}
\url{https://github.com/aghitza/kurihara_numbers}.
\end{center}
Although the original code is for good reduction, a slight modification allows us to compute the additive reduction case.
%
%
%
%

\subsection{Elliptic curve of conductor 56144}
Let $p = 11$ and  $E$ be the elliptic curve over $\mathbb{Q}$ defined by
$$y^2=x^3-584551x-172021102$$
with conductor $56144 =  2^4 \cdot 11^2 \cdot 29$ as in \cite[\href{http://www.lmfdb.org/EllipticCurve/Q/56144/w/1}{Elliptic Curve 56144.w1}]{lmfdb}.
The SAGE computation yields the following facts:
\begin{itemize}
\item The mod 11 representation is surjective;
\item The Tamagawa factor of $E$ is prime to $11$;
\item $a_{29}(E) = 1$, so 11 does not divide $28 = 29-1$.
\item $\widetilde{\delta}_{397 \cdot 859} \neq 0$ where 397 and 859 are (the smallest) Kolyvagin primes for $(E, p)$.
\end{itemize}
Thus, Theorem \ref{thm:main_thm_iwasawa} and Theorem \ref{thm:main_thm_bsd} for $(E, p)$ hold.
Note that these theorems do not directly follow from the one divisibility since $\#\sha(E/\mathbb{Q})[11^\infty]$ is 121.
\subsection{Elliptic curve of conductor 84700}
Let $p = 11$ and  $E$ be the elliptic curve over $\mathbb{Q}$ defined by
$$y^2 = x^3 - 235390375x - 1409480751250$$
with conductor $84700 = 2^2 \cdot 5^2 \cdot 7 \cdot 11^2$ as in \cite[\href{http://www.lmfdb.org/EllipticCurve/Q/84700/bi/1}{Elliptic Curve 84700.bi1}]{lmfdb}.
The SAGE computation yields the following facts:
\begin{itemize}
\item The mod 11 representation is surjective;
\item The Tamagawa factor is prime to $11$;
\item $a_{7}(E) = -1$, so 11 does not divide $8 = 7 + 1$;
\item $\widetilde{\delta}_{23 \cdot 2113} \neq 0$ where 23 and 2113 are (the smallest) Kolyvagin primes for $(E, p)$.
\end{itemize}
Thus, Theorem \ref{thm:main_thm_iwasawa} and Theorem \ref{thm:main_thm_bsd} for $(E, p)$ hold.
Note that these results do not directly follow from the one divisibility since $\#\sha(E/\mathbb{Q})[11^\infty]$ is 121.
\subsection{Elliptic curve of conductor 84100}
Let $p = 5$ and  $E$ be the elliptic curve over $\mathbb{Q}$ defined by
$$y^2=x^3+x^2-28033x+1232688$$
with conductor $84100 = 2^2 \cdot 5^2 \cdot 29^2$ as in \cite[\href{http://www.lmfdb.org/EllipticCurve/Q/84100/b/3}{Elliptic Curve 84100.b3}]{lmfdb}.
The SAGE computation yields the following facts:
\begin{itemize}
\item The mod 5 representation is surjective;
\item $a_4 = 28033 \not\equiv 10 \Mod{25}$, so it is not an exceptional case (Assumption \ref{assu:exceptional_cases});
\item The Tamagawa factor is prime to $5$;
\item $\widetilde{\delta}_{191 \cdot 331} \neq 0$ where 191 and 331 are (the smallest) Kolyvagin primes for $(E, p)$.
\end{itemize}
Considering Remark \ref{rem:exceptional_case}.(1), Theorem \ref{thm:main_thm_iwasawa} and Theorem \ref{thm:main_thm_bsd} hold for $(E,p)$.
Note that these results do not directly follow from the one divisibility since the rank of $E$ is two.

\section*{Acknowledgement}
This project grew out from the discussion when we both visit Shanghai Center for Mathematical Sciences in March 2018.
We deeply thank Shanwen Wang and the generous hospitality of Shanghai Center for Mathematical Sciences.
We deeply appreciate Michiel Kosters to allow us to reproduce some material in \cite{kosters-pannekoek} in $\S$\ref{sec:computing-lattice}.
C.K. was partially supported by Basic Science Research Program through the National Research Foundation of Korea (NRF-2018R1C1B6007009).
K.N. was supported by JSPS KAKENHI Grant Number 90595993.

\appendix
\section{The choice of Kato's Euler system} \label{sec:choices}
The goal of this appendix is to give an \emph{explicit} characterization of the ``optimal" Kato's Euler system
$$c_{F} \in \mathrm{H}^1(F, T(-1))$$ where $F$ runs over abelian extensions of $\mathbb{Q}$ such that $$\mathrm{exp}^* (c_{\mathbb{Q}}(1)) = \frac{L^{(Np)}(E,1)}{\Omega^+_E}$$
where $c_{\mathbb{Q}}(1)$ is the image of $c_{\mathbb{Q}}$ in $\mathrm{H}^1(\mathbb{Q}, T)$ under the map induced from the Tate twist as described in \cite[Theorem 12.5]{kato-euler-systems}. This choice of Kato's Euler system is used in many literatures; however, it seems a bit difficult to find the explicit choice from the original construction of Kato in the literatures.
Note that the Euler system in \cite{kato-euler-systems} is constructed for $T(-1)$ due to the cohomological convention of the Galois representation (\hspace{1sp}\cite[$\S$14.10]{kato-euler-systems}).

Following \cite[Corollary 7.2]{rubin-es-mec} and \cite{scholl-kato}, there exists an Euler system $\lbrace c'_F(1) \rbrace_F$ for $T$ such that
$$\mathrm{exp}^*(c'_{\mathbb{Q}} (1)  ) = r_E \cdot \frac{L^{(Np)}(E,1)}{\Omega^+_E}$$
for some positive integer $r_E$.
The goal of this appendix is to show that we can take $r_E = 1$ following \cite[Theorem 6.6]{kato-euler-systems} and \cite[Remark in Appendix A (Page 254)]{delbourgo-book} choosing the ``right" Kato's Euler system.

We recall the convention for Kato's Euler systems. We assume some familiarity with \cite{kato-euler-systems} here. We sometimes may use the notation therein from without any caution.
\subsection{The cohomology classes} \label{subsec:the_cohomology_classes}
Let $f \in S_k(\Gamma_1(N), \psi)$ be a newform with Hecke field $F$ over $\mathbb{Q}$ and $S(f)$ be the rank one Hecke module generated by $f$ over $F$ as in \cite[$\S$6.3]{kato-euler-systems}. Let $\lambda$ be the place of $F$ compatible with a choice of embedding $\iota_p : \overline{\mathbb{Q}} \to \overline{\mathbb{Q}}_p$.
Let $m \geq 1$ be an integer.
\begin{defn} \label{defn:xi_S}
Following \cite[$\S$5.1]{kato-euler-systems}, we define $\xi$ and $S$ as follows:
\begin{enumerate}
\item $\xi$ is a symbol $a(A)$ where $a, A \in \mathbb{Z}$, $A \geq 1$ and $S$ is a non-empty finite set of primes containing $\mathrm{prime}(mA)$, or
\item $\xi$ is an element of $\mathrm{SL}_2(\mathbb{Z})$ and $S$ is a non-empty finite set of primes containing $\mathrm{prime}(mN)$.
\end{enumerate}
\end{defn}
\begin{defn} \label{defn:r_r'_c_d}
Following \cite[(5.2.1) and (5.2.2)]{kato-euler-systems}, we define integers $r, r'$ and positive integers $c, d$ as follows:
\begin{enumerate}
\item $1 \leq r \leq k-1$, $1 \leq r' \leq k-1$, and at least one of $r$, $r'$ is $k-1$.
\item $c$ and $d$ are positive integers with $\mathrm{prime}(cd) \cap S = \emptyset$, and $(d,N) = 1$.
\end{enumerate}
\end{defn}

\begin{defn} \label{defn:zeta_modular_forms}
Let 
$${}_{c,d}z_m(f,r,r', \xi, S) \in S(f) \otimes \mathbb{Q}(\mu_m)$$
be the zeta modular form in \cite[$\S$6.3]{kato-euler-systems}
where $c$, $d$, $r$, $r'$, $\xi$, and $S$ satisfy Definition \ref{defn:xi_S} and Definition \ref{defn:r_r'_c_d}.
\end{defn}

For any $x \in V_{F}(f)$, $x^{\pm} := \frac{1}{2} (1\pm \iota) x$ where $\iota$ is induced from the complex conjugation.
Let  $V_{F}(f)^\pm :=  V_{F}(f)^{\iota = \pm 1}$ and then $\mathrm{dim}_F \ V_{F}(f)^\pm = 1$, respectively.
Let $\mathrm{per}_f : S(f) \to V_{F}(f) \otimes_{F} \mathbb{C}$ be the period map
in \cite[$\S$6.3]{kato-euler-systems}.
Let $f^*$ be the complex conjugation of $f$ (``the dual modular form").
\begin{thm}[{\hspace{1sp}\cite[Theorem 6.6]{kato-euler-systems}}] \label{thm:kato_interpolation_original}
Let $\chi$ be a character on $(\mathbb{Z}/m\mathbb{Z})^\times$.
Let $\pm = (-1)^{k-r-1} \cdot \chi(-1)$ and 
 $$(u,v) := \left\lbrace \begin{array}{ll} (r+2-k, r) & \textrm{if } r'=k-1 \\ (k-r', r') & \textrm{if } r=k-1  \end{array} \right.$$
 following \cite[(4.2.4)]{kato-euler-systems}.
We add the following condition.
\begin{enumerate}
\item[($*$)] If $\xi \in \mathrm{SL}_2(\mathbb{Z})$, then assume $c \equiv d \equiv 1 \Mod{N}$.
\end{enumerate}
Then we have
$$\sum_{b \in (\mathbb{Z}/m\mathbb{Z})^\times} \chi(b) \cdot \mathrm{per}_f \left( \sigma_b \left( {}_{c,d}z_m(f, r, r', \xi, S) \right) \right)^{\pm} = L^{(S)} (f^*, \chi, r) \cdot (2 \pi i )^{k-r-1} \cdot \gamma^{\pm}$$
where
\begin{align*}
\gamma = c^2 d^2 \delta(f, r', a(A)) & - c^u d^2 \overline{\chi}(c) \delta(f, r', ac(A)) \\
& - c^2 d^v \psi(d) \delta(f, r', ``a/d"(A)) + c^u d^v \overline{\chi}(cd) \psi(d)\delta(f, r', ``ac/d"(A))
\end{align*}
if $\xi = a(A)$, and
$$\gamma = (c^2 - c^u \overline{\chi}(c)) (d^2 - d^v \overline{\chi}(d)) \delta(f, r', \xi)$$
if $\xi \in \mathrm{SL}_2(\mathbb{Z})$.
\end{thm}

\begin{defn} \label{defn:zeta_euler_systems}
Let
$${}_{c,d}z^{(p)}_m(f,r,r', \xi, S) \in \mathrm{H}^1(\mathbb{Z}[1/p, \zeta_m], V_{\mathcal{O}_\lambda}(f)(k-r))$$
be the zeta element in \cite[(8.1.3)]{kato-euler-systems}
where $p \in S$ as well as all the ``variables" satisfy Definition \ref{defn:xi_S} and Definition \ref{defn:r_r'_c_d}.
\end{defn}
By \cite[Proposition 8.12]{kato-euler-systems}, the cohomology classes 
$${}_{c,d}z^{(p)}_m(f,r,r', \xi, S) \in \mathrm{H}^1(\mathbb{Z}[1/p, \zeta_m], V_{\mathcal{O}_\lambda}(f)(k-r))$$ satisfy the Euler system relation except at primes dividing $cdN$. See also \cite[$\S$13.3]{kato-euler-systems}.
By \cite[Theorem 9.7]{kato-euler-systems}, 
under Definition \ref{defn:r_r'_c_d}.(1), we have the following assignment
$${}_{c,d}z^{(p)}_m(f,r,r', \xi, S) \mapsto {}_{c,d}z_m(f,r,r', \xi, S)$$
under the dual exponential map. In other words, zeta modular forms (Definition \ref{defn:zeta_modular_forms}) and zeta elements (Definition \ref{defn:zeta_euler_systems}) match in the \emph{critical} range, i.e. $L$-values at $s =1, \cdots , k-1$.

\subsection{Kato's modification in the $p$-power direction} \label{subsec:kato_modification}
Following \cite[Theorem 12.5 and $\S$13.9]{kato-euler-systems}, we recall Kato's modification of zeta elements (in the $p$-power direction) in order to specify $\mathbf{z}_{\mathrm{Kato}}$ in Conjecture \ref{conj:kato-main-conjecture}.
\subsubsection{The elements}
Let $G_\infty :=\mathrm{Gal}(\mathbb{Q}(\zeta_{p^\infty})/\mathbb{Q})$ and
 $j_{p^n} : \mathrm{Spec}(\mathbb{Q}(\zeta_{p^n})) \to \mathrm{Spec}(\mathbb{Z}[\zeta_{p^n}, 1/p])$ be the natural map.
Then we define a $\mathbb{Z}_p\llbracket G_\infty \rrbracket$-module
$$\mathbf{H}^i(V_{\mathcal{O}_\lambda}(f)) := \varprojlim_{n} \mathrm{H}^i_{\mathrm{\acute{e}t}}( \mathrm{Spec}(\mathbb{Z}[\zeta_{p^n}, 1/p], j_{p^n, *}V_{\mathcal{O}_\lambda}(f)) $$
where $\mathrm{H}^i_{\mathrm{\acute{e}t}}( \mathrm{Spec}(\mathbb{Z}[\zeta_{p^n}, 1/p]), j_{p^n, *}V_{\mathcal{O}_\lambda}(f))$ is the \'{e}tale cohomology group.
\begin{choice} \label{choice:alpha_j}
Fix elements $\alpha_1, \alpha_2 \in \mathrm{SL}_2(\mathbb{Z})$ and integers $j_1, j_2$ such that
\begin{itemize}
\item  $1 \leq j_i \leq k-1$ $(i=1, 2)$, and
\item $\delta(f , j_1, \alpha_1)^+ \neq 0$ and $\delta(f , j_2, \alpha_2)^- \neq 0$ (\hspace{1sp}\cite[(13.6)]{kato-euler-systems}).
\end{itemize}
\end{choice}
Let $\gamma \in V_{F_{\lambda}}(f)$.
Then we have
\begin{equation} \label{eqn:gamma}
\gamma = b_1 \cdot \delta(f , j_1, \alpha_1)^+ + b_2 \cdot \delta(f , j_2, \alpha_2)^- \neq 0
\end{equation}
for some $b_1, b_2 \in \mathbb{Q}_{f,\lambda}$.
\begin{choice} \label{choice:c_and_d}
Fix $c, d \in \mathbb{Z}$ such that
\begin{enumerate}
\item $(cd, 6p) = 1$,
\item $c \equiv d \equiv 1 \Mod{N}$ (needed for \cite[Theorem 6.6]{kato-euler-systems}), and
\item $c^2 \neq 1$ and $d^2 \neq 1$.
\end{enumerate}
\end{choice}
For a commutative ring $R$, let $Q(R)$ be the total quotient ring of $R$.
\begin{defn} \label{defn:kato_modification}
We define
\begin{align*}
\mathbf{z}^{(p)}_{\gamma} := &
\left\lbrace \mu(c,d,j_1)^{-1} \cdot b_1 \cdot \left( {}_{c,d} \mathbf{z}^{(p)}_{p^n}(f,k, j_1, \alpha_1, \mathrm{prime}(pN)) \right)_{n \geq 1} \right\rbrace^{-} \\
& +
\left\lbrace \mu(c,d,j_2)^{-1} \cdot b_2 \cdot \left( {}_{c,d} \mathbf{z}^{(p)}_{p^n}(f,k, j_2, \alpha_2, \mathrm{prime}(pN)) \right)_{n \geq 1} \right\rbrace^{+} 
\in  \mathbf{H}^1(V_{\mathcal{O}_{\lambda}}(f)) \otimes Q(\mathbb{Z}_p\llbracket G_\infty \rrbracket)
\end{align*}
where $$\mu(c,d,j) := (c^2 - c^{k+1-j} \cdot \sigma_c) \cdot (d^2 - d^{j+1} \cdot \sigma_d) \cdot \prod_{\ell \mid N} (1 - \overline{a_\ell(f)} \ell^{-k} \sigma^{-1}_{\ell}) \in \mathbb{Z}_p\llbracket G_\infty \rrbracket $$
for $j \in \mathbb{Z}$ and $\mu(c,d,j)$ is not a zero divisor in $\mathbb{Z}_p\llbracket G_\infty \rrbracket$ for all $j$.
\end{defn}
By \cite[$\S$13.12]{kato-euler-systems}, indeed, it is known that $\mathbf{z}^{(p)}_{\gamma} \in  \mathbf{H}^1(V_{\mathcal{O}_{\lambda}}(f)) \otimes \mathbb{Q}_p$.

\subsubsection{The submodule inverting $p$}
Let $Z(f)$ be the $\mathcal{O}_\lambda\llbracket G_\infty\rrbracket \otimes \mathbb{Q}_p$-submodule of $\mathbf{H}^1(V_{\mathcal{O}_{\lambda}}(f)) \otimes \mathbb{Q}_p$ generated by $\mathbf{z}^{(p)}_{\gamma}$ for all $\gamma \in V_{F_{\lambda}}(f)$.
Since $\mathbf{H}^1(V_{\mathcal{O}_{\lambda}}(f)) \otimes \mathbb{Q}_p$ is free of rank one over $\mathcal{O}_\lambda\llbracket G_\infty \rrbracket \otimes \mathbb{Q}_p$ (by \cite[Theorem 12.4.(2)]{kato-euler-systems}), we easily observe the following independence result considering the evaluation at all the finite order characters on $G_\infty$.
\begin{prop} \label{prop:independence_z_kato}
The cohomology class $\mathbf{z}^{(p)}_{\gamma}$ is independent of the choices of $\alpha_1, \alpha_2, j_1, j_2, c, d$ as above, i.e. Choice \ref{choice:alpha_j} and Choice \ref{choice:c_and_d}.
\end{prop}
Furthermore, since $\mathcal{O}_\lambda\llbracket G_\infty \rrbracket \otimes \mathbb{Q}_p$ is a product of PIDs, there exists an element $\gamma_0 \in V_f$ such that $\mathbf{z}^{(p)}_{\gamma_0}$ generates $Z(f)$ over $\mathcal{O}_\lambda \llbracket G_\infty \rrbracket \otimes \mathbb{Q}_p$. 
Thus, any non-zero $\mathbf{z}^{(p)}_{\gamma}$ generates $Z(f)$ over $\mathcal{O}_\lambda\llbracket G_\infty \rrbracket \otimes \mathbb{Q}_p$. 
\subsubsection{The submodule without inverting $p$}
We assume the following condition and it is satisfied in most cases as described in \cite[Remark 12.8]{kato-euler-systems}.
\begin{assu} \label{assu:big_image}
Suppose that there exists a $\mathcal{O}_{\lambda}$-lattice $T_f \subseteq V_{F_\lambda}(f)$ such that the image of
$\rho_{T_f} : \mathrm{Gal}(\overline{\mathbb{Q}}/\mathbb{Q}(\mu_{p^\infty})) \to \mathrm{GL}_{\mathcal{O}_{\lambda}}(T_f)$ contains $\mathrm{SL}_2(\mathbb{Z}_p)$.
\end{assu}
Let $Z(f,T_f)$ be the $\mathbb{Z}\llbracket G_\infty \rrbracket$-submodule of $\mathbf{H}^1(T_f) \otimes \mathbb{Q}_p = \mathbf{H}^1(V_{\mathcal{O}_{\lambda}}(f)) \otimes \mathbb{Q}_p$ generated by $\mathbf{z}^{(p)}_{\gamma}$ for all $\gamma \in T_f \subseteq V_{F_{\lambda}}(f)$.
Then by \cite[Theorem 12.5.(4)]{kato-euler-systems}, it is known that $Z(f, T_f) \subseteq \mathbf{H}^1(T_f)$ under Assumption \ref{assu:working_assumptions}.(2).

Suppose that $T_f = V_{\mathcal{O}_{\lambda}}(f)$. Let $Z$ be the $\mathbb{Z}_p\llbracket G_\infty \rrbracket$-submodule of $Z(f, V_{\mathcal{O}_{\lambda}}(f))$ generated by
$${}_{c,d} \mathbf{z}^{(p)}_{p^n}(f,k, j, a(A), \mathrm{prime}(pA))_{n \geq 1} \in \mathbf{H}^1( V_{\mathcal{O}_{\lambda}}(f) )$$
where $1 \leq j \leq k-1$, $a, A \in \mathbb{Z}$ with $A \geq 1$, $c, d \in \mathbb{Z}$ such that $(c, 6pA)= (d, 6pN) =1$, and
$${}_{c,d} \mathbf{z}^{(p)}_{p^n}(f,k, j, \alpha, \mathrm{prime}(pN))_{n \geq 1}  \in \mathbf{H}^1( V_{\mathcal{O}_{\lambda}}(f) ) $$
where $1 \leq j \leq k-1$, $\alpha \in \mathrm{SL}_2(\mathbb{Z})$, $c, d \in \mathbb{Z}$ such that $(c, 6pA)= (d, 6pN) =1$.
Then $Z \subseteq Z( f ,V_{\mathcal{O}_{\lambda}}(f) )$ of finite index due to \cite[Theorem 12.6]{kato-euler-systems}.
\begin{rem} \label{rem:comparing_lattices}
Indeed, Assumption \ref{assu:big_image} is satisfied for all lattices if it is satisfied for one lattice for non-CM elliptic curves. See \cite[Remark 12.8]{kato-euler-systems}.
\end{rem}
\begin{prop}
Under Assumption \ref{assu:big_image}, 
the submodule $Z(f,T_f) \subseteq \mathbf{H}^1(T_f)$ is generated by one element over $\mathcal{O}_\lambda\llbracket G_\infty \rrbracket$.
\end{prop}
\begin{proof}
By Remark \ref{rem:comparing_lattices}, we may assume $V_{\mathcal{O}_\lambda}(f) = T_f$.
Let $\alpha : V_{\mathcal{O}_\lambda}(f) \to \mathbf{H}^1(V_{\mathcal{O}_\lambda}(f))$ be the map defined by $\gamma \mapsto \mathbf{z}^{(p)}_{\gamma}$ as in \cite[Theorem 12.5.(1)]{kato-euler-systems}.
Consider the $\mathcal{O}_\lambda\llbracket G_\infty \rrbracket$-linear map
$$\beta :
V_{\mathcal{O}_\lambda}(f)^+ \otimes_{\mathcal{O}_\lambda} \mathcal{O}_\lambda\llbracket G_\infty \rrbracket^- \oplus
V_{\mathcal{O}_\lambda}(f)^- \otimes_{\mathcal{O}_\lambda} \mathcal{O}_\lambda\llbracket G_\infty \rrbracket^+ \to \mathbf{H}^1(V_{\mathcal{O}_\lambda}(f))$$
defined by
$\gamma^+ \otimes \lambda_1 +
\gamma^- \otimes \lambda_2 \mapsto 
 \lambda_1 \cdot \mathbf{z}^{(p)}_{\gamma^+}+
 \lambda_2 \cdot \mathbf{z}^{(p)}_{\gamma^-}$
and then it satisfies $\mathcal{O}_\lambda\llbracket G_\infty \rrbracket \cdot \mathrm{Im}(\alpha) = \mathrm{Im}(\beta)$.
The map $\beta$ is well-defined due to the equation
 $$\mathbf{z}^{(p)}_{\iota(\gamma)} = - \sigma_{-1} \left( \mathbf{z}^{(p)}_{\gamma} \right) .$$
Since 
$V_{\mathcal{O}_\lambda}(f)^+ \otimes_{\mathcal{O}_\lambda} \mathcal{O}_\lambda\llbracket G_\infty \rrbracket^-$ and
$V_{\mathcal{O}_\lambda}(f)^- \otimes_{\mathcal{O}_\lambda} \mathcal{O}_\lambda\llbracket G_\infty \rrbracket^+$ 
are free of rank one over
$\mathcal{O}_\lambda\llbracket G_\infty \rrbracket^-$ and $\mathcal{O}_\lambda\llbracket G_\infty \rrbracket^+$,
respectively, $\mathrm{Im}(\beta)$ is free or rank one over $ \mathcal{O}_\lambda\llbracket G_\infty \rrbracket$. 
Thus, $\mathcal{O}_\lambda\llbracket G_\infty \rrbracket \cdot \mathrm{Im}(\alpha)$ is also free or rank one over $ \mathcal{O}_\lambda\llbracket G_\infty \rrbracket$; hence, the conclusion holds.
\end{proof}

\begin{defn}[Kato's zeta element for elliptic curves] \label{defn:z_kato}
Assume that $f$ corresponds to an elliptic curve over $\mathbb{Q}$ and $T$ be its $p$-adic Tate module.
Let 
$\mathbf{z}^{(p)}_{\gamma_0}$ be a generator of $Z(f, V_{\mathcal{O}_\lambda}(f)) \subseteq \mathbf{H}^1(V_{\mathcal{O}_\lambda}(f))$ over $ \mathcal{O}_\lambda\llbracket G_\infty \rrbracket$. Then we define $\mathbf{z}_{\mathrm{Kato}} \in \mathbb{H}^1(T)$ in Conjecture \ref{conj:kato-main-conjecture}  by
the image of $\mathbf{z}^{(p)}_{\gamma_0}$ under the map
\[
\xymatrix{
\mathbf{H}^1(V_{\mathcal{O}_\lambda}(f)) \ar[r]^-{\otimes (\zeta_{p^n})_n} & \mathbf{H}^1(V_{\mathcal{O}_\lambda}(f)(1)) \ar[r]^-{\mathrm{cores}} & \mathbb{H}^1(T) 
}
\]
where the corestriction map is from $\mathbb{Q}(\mu_{p^\infty})$ to $\mathbb{Q}_{\infty}$.
\end{defn}
\begin{rem}
Note that $\gamma^{\pm}_0$ corresponds to the N\`{e}ron periods of the elliptic curve since $T$ comes from the $p$-adic Tate module of $E$.
\end{rem}
\subsection{The Euler system: an extension to the tame direction} \label{subsec:extension_of_euler_systems}
The goal of this section is to extend the interpolation property of Kato's modification $\mathbf{z}^{(p)}_\gamma$ obtained from 
\begin{quote}
$\left( {}_{c,d} \mathbf{z}^{(p)}_{p^n}(f,k, j_1, \alpha_1, \mathrm{prime}(pN)) \right)_{n \geq 1}$ and
$\left( {}_{c,d} \mathbf{z}^{(p)}_{p^n}(f,k, j_2, \alpha_2, \mathrm{prime}(pN)) \right)_{n \geq 1}$
\end{quote}
to the tame direction \emph{very slightly}.
The motivation of the extension is the following question.
\begin{ques} \label{ques:ideal_euler_systems}
Do we have an \emph{integral} Euler system for the triple $( V_{\mathcal{O}_{\lambda}}(f)(k-r), F_\lambda, \mathrm{prime}(cdpN) )$ modifying ${}_{c,d}z^{(p)}_m(f,r,j, \xi, \mathrm{prime}(mN)) \in \mathrm{H}^1(\mathbb{Z}[1/p, \zeta_m], V_{\mathcal{O}_{\lambda}}(f)(k-r))$ such that
\begin{enumerate}
\item it is \emph{independent} of the choices of $(c, d)$ and $(j, \xi)$ (Proposition \ref{prop:independence_z_kato}) and
\item it satisfies the interpolation property for \emph{all finite order characters on $\mathrm{Gal}(\mathbb{Q}^{\mathrm{ab}}/\mathbb{Q})$} (Theorem \ref{thm:kato_interpolation_original})?
\end{enumerate}
\end{ques}
We do not have the answer to this question yet. In \cite[Corollary 7.2]{rubin-es-mec}, when $k = 2$, $r=1$ and $j_1=j_2=1$, Rubin could modify the Euler system 
$${}_{c,d}z^{(p)}_m(f,1,1, \xi, \mathrm{prime}(mpN))$$
 by choosing $c$ and $d$ more carefully, i.e. $c \not\equiv 1 \Mod{p}$ and $d \not\equiv 1 \Mod{p}$ in order to have 
 $$c \cdot d \cdot (c - \sigma_c) \cdot (d - \sigma_d) \in \Lambda^\times .$$ Thus, Rubin's modification yields both the integrality and the interpolation property for finite order characters on $\mathrm{Gal}(\mathbb{Q}_{\infty}/\mathbb{Q})$.
We extend Rubin's strategy to the case we need for Theorem \ref{thm:kato_interpolation} using the following lemma.
%

\begin{lem} \label{lem:invert_c_d}
Assume that
\begin{enumerate} 
\item[$(**)$]  $j = k-1$, $k<p$, and $r \neq 2$.
\end{enumerate}
Then there exist infinitely many integers $c$ and $d$ such that
$$(c^2 - c^{r+1-j} \cdot \sigma_c) \cdot (d^2 - d^{j+1 + r - k} \cdot \sigma_d ) \in \mathbb{Z}_p[\mathrm{Gal}(F/\mathbb{Q})]^\times .$$
where $F$ is a cyclic extension of $\mathbb{Q}$ of $p$-power degree ramified only at a prime $\ell$ with $\ell \equiv 1 \Mod{p}$ and $(\ell, cdp)=1$.
\end{lem}
\begin{proof}
Since
\begin{align*}
& (c^2 - c^{r+1-j} \cdot \sigma_c) \cdot (d^2 - d^{j+1 + r - k} \cdot \sigma_d ) \\
 = & c^{r+1-j} \cdot (c^{1-r+j} -1 + 1 - \sigma_c ) \cdot d^{j+1+r-k} \cdot (d^{1-j-r+k} -1 + 1 - \sigma_d) \\
 = & c^{r+1-k+1} \cdot (c^{1-r+k-1} -1 + 1 - \sigma_c ) \cdot d^{k-1+1+r-k} \cdot (d^{1-k+1-r+k} -1 + 1 - \sigma_d) \\
 = & c^{r-k+2} \cdot (c^{k-r} -1 + 1 - \sigma_c ) \cdot d^{r} \cdot (d^{2-r} -1 + 1 - \sigma_d),
\end{align*}
under Assumption $(**)$, it suffices to choose $c$ and $d$ such that their reductions mod $p$ are primitive roots mod $p$. 
\end{proof}
\begin{rem} \label{rem:invert_c_d}
Assumption $(**)$ is already observed in \cite[(5.2.3)]{kato-euler-systems}. If $k=2$, Assumption $(**)$ in Lemma \ref{lem:invert_c_d} is automatic. 
\end{rem}
Following \cite[$\S$13.3]{kato-euler-systems}, we consider the following Euler system 
$${}_{c,d}z^{(p)}_m(f,r, k-1, \xi, \mathrm{prime}(mpN)) $$
in $\mathrm{H}^1(\mathbb{Z}[1/p, \zeta_m], V_{\mathcal{O}_{\lambda}}(f)(k-r))$
for the triple $( V_{\mathcal{O}_{\lambda}}(f)(k-r), F_\lambda, \mathrm{prime}(cdpN) )$
where $\xi \in \mathrm{SL}_2(\mathbb{Z})$.

\begin{rem} $ $
\begin{enumerate}
\item
Note that we have the following equality at the level of Iwasawa cohomology
\begin{align*}
& \left( {}_{c,d}\mathbf{z}^{(p)}_{mp^n}(f,r, k-1, \xi, \mathrm{prime}(mpN)) \right)_n \\
= & \left( {}_{c,d}\mathbf{z}^{(p)}_{mp^n}(f,k,k-1, \xi, \mathrm{prime}(mpN)) \right)_n \otimes \left( \zeta^{\otimes k-r}_{p^n} \right)_n
\end{align*}
in $\varprojlim_n \mathrm{H}^1(\mathbb{Z}[1/p, \zeta_{mp^n}], V_{\mathcal{O}_{\lambda}}(f)(k-r))$.
\item
Since we do not expect
$$ \displaystyle \prod_{\ell \mid N} (1 - \overline{a_\ell(f)} \ell^{-r} \sigma^{-1}_{\ell}) $$
is invertible in general, we do not invert these Euler factors at bad primes.
\end{enumerate}
\end{rem}

Under Assumption $(**)$ in Lemma \ref{lem:invert_c_d},
let
$\nu(c,d,k-1, k-r) $
be \emph{any} element of $\mathbb{Z}_p\llbracket \mathrm{Gal}(\mathbb{Q}^{\mathrm{ab}}/\mathbb{Q}) \rrbracket$  which restricts to
$$ \left( (c^2 - c^{r+1-j} \cdot \sigma_c) \cdot (d^2 - d^{j+1 + r - k} \cdot \sigma_d ) \right)^{-1}  \in \mathbb{Z}_p\llbracket \mathrm{Gal}(\widetilde{F}/\mathbb{Q}) \rrbracket $$
where $\widetilde{F}$ is the compositum of all cyclic extensions of $\mathbb{Q}$ of $p$-power degree ramified only at a prime $\ell$ with $(\ell, pcdN)=1$ and $\mathbb{Q}_\infty$.
Then we define
\begin{align*}
z^{(Np)}_{\gamma, m}(k-r) := &
\left\lbrace \nu(c,d,k-1, k-r) \cdot b_1 \cdot  {}_{c,d} z^{(p)}_{m}(f,r, k-1, \alpha_1, \mathrm{prime}(pNm))  \right\rbrace^{-} \\
& +
\left\lbrace \nu(c,d,k-1, k-r) \cdot b_2 \cdot  {}_{c,d} z^{(p)}_{m}(f,r, k-1, \alpha_2, \mathrm{prime}(pNm))  \right\rbrace^{+} 
\end{align*}
in $\mathrm{H}^1(\mathbb{Z}[1/p, \zeta_{m}], V_{\mathcal{O}_{\lambda}}(f)(k-r))$ and then $z^{(Np)}_{\gamma, m}(k-r)$ obviously forms an integral Euler system and interpolates $Np$-imprimitive $L$-values of $f$ at $s=r$ twisted by finite order characters on $\mathrm{Gal}(\widetilde{F}/\mathbb{Q}) $.
Since the choice of $\nu(c,d,k-1, k-r)$ is arbitrary, we do not know the precise interpolation formula for general finite order characters on $\mathrm{Gal}(\mathbb{Q}^{\mathrm{ab}}/\mathbb{Q})$.

\subsection{Optimization of periods}
The goal here is to make a \emph{good} choice of $\gamma$ in $z^{(Np)}_{\gamma, m}(k-r)$. 

We recall that the period map appeared in $\S$\ref{subsec:the_cohomology_classes} induces the Shimura isomorphism
\begin{equation} \label{eqn:shimura_isom}
\mathrm{per}_f: \left( S(f) \otimes \mathbb{C} \right)^2 \simeq V_F(f) \otimes \mathbb{C}
\end{equation}
defined by $(f, f) \mapsto \mathrm{per}_f(f) + \iota' \mathrm{per}_f(f)$
where $\iota' (x \otimes y ) = x \otimes \overline{y}$ for $x \in V_F(f)$, $y \in \mathbb{C}$. See \cite[(7.13.2)]{kato-euler-systems} for detail.

\subsubsection{N{\'{e}}ron periods}
Assume that $f$ corresponds to an elliptic curve $E$ over $\mathbb{Q}$.
Let $\omega_f \in S(f)$ be the element corresponding to the N\'{e}ron differential $\omega_E$ of $E$ under the modular parametrization.
Then we have
$$\mathrm{per}_f (\omega_f) = \Omega^+_E \cdot \gamma^+ + \Omega^-_E \cdot \gamma^- $$
for some non-zero $\gamma^\pm \in V_{\mathbb{Z}}(f)$.
Here, $\gamma^\pm \in V_{\mathbb{Z}}(f) \subseteq T(-1)$ is due to
Theorem \ref{thm:integrality_modular_symbols} and \cite[Theorem 13.6]{kato-euler-systems} (c.f. \cite{ash-stevens}).
\begin{choice} \label{choice:gamma}
We choose $b_1, b_2 \in \mathbb{Q} (= \mathbb{Q}_{f,\lambda})$ such that
\[
\xymatrix{
\gamma^+ = b_1 \cdot \delta(f , 1, \alpha_1)^+ , & \gamma^- =   b_2 \cdot \delta(f , 1, \alpha_2)^- 
}
\]
in Equation (\ref{eqn:gamma}).
\end{choice}

\begin{defn} \label{defn:the_defn_ES}
Due to Lemma \ref{lem:invert_c_d} with Remark \ref{rem:invert_c_d}, the Euler system in $\S$\ref{subsec:euler_systems_modular_symbols} can be chosen by
 $$c_{\mathbb{Q}(\mu_m)}(1) := z^{(Np)}_{\gamma, m}(1)$$
where $f$ is the modular form corresponding to $E$, $k = 2$, $j_1 = j_2 = 1$, and $\alpha_1$, $\alpha_2$ and 
$\gamma^{\pm}$ follows Choice \ref{choice:gamma}.
\end{defn}
We also consider the Iwasawa cohomology version. Let
\begin{align*}
\mathbf{z}^{(Np)}_{\gamma} := &
 \left\lbrace \nu(c,d,1,1) \cdot b_1 \cdot \left( {}_{c,d} \mathbf{z}^{(Np)}_{p^n}(f,1, 1, \alpha_1, \mathrm{prime}(pNm)) \right)_{n \geq 1} \right\rbrace^{-}  \\
& +
 \left\lbrace \nu(c,d,1,1) \cdot b_2 \cdot \left( {}_{c,d} \mathbf{z}^{(Np)}_{p^n}(f,1, 1, \alpha_2, \mathrm{prime}(pNm)) \right)_{n \geq 1} \right\rbrace^{+}  .
\end{align*}
Then we define an $N$-imprimitive analogue of $\mathbf{z}_{\mathrm{Kato}}$ (Definition \ref{defn:z_kato}) for elliptic curves by
\begin{equation} \label{eqn:imprimitive_z_kato}
\mathbf{z}^{(N)}_{\mathrm{Kato}} := \mathrm{cores}_{\mathbb{Q}(\mu_{p^\infty})/\mathbb{Q}_\infty}\left( \mathbf{z}^{(Np)}_{\gamma_0} (1) \right) .
\end{equation}
\begin{lem} \label{lem:comparing_main_conj_imprimitive}
Under Assumption \ref{assu:working_assumptions}.(1), we have
$$\mathrm{char}_{\Lambda} \left( \dfrac{ \mathbb{H}^1(T) }{ \Lambda \mathbf{z}^{(N)}_{\mathrm{Kato}} } \right)
=
\mathrm{char}_{\Lambda} \left( \dfrac{ \mathbb{H}^1(T) }{ \Lambda \mathbf{z}_{\mathrm{Kato}} } \right) .$$
\end{lem}
\begin{proof}
By definition, we have
$$\mathbf{z}^{(N)}_{\mathrm{Kato}} = 
\prod_{\ell \mid N_{\mathrm{st}}} (1 - \ell^{-1} \sigma^{-1}_\ell) \cdot \prod_{\ell \mid N_{\mathrm{ns}}} (1 + \ell^{-1} \sigma^{-1}_\ell) \cdot
\mathbf{z}_{\mathrm{Kato}} 
 ,$$
and, indeed,
$$\prod_{\ell \mid N_{\mathrm{st}}} (1 - \ell^{-1} \sigma^{-1}_\ell) \cdot \prod_{\ell \mid N_{\mathrm{ns}}} (1 + \ell^{-1} \sigma^{-1}_\ell) $$ is invertible in $\Lambda$ under Assumption \ref{assu:working_assumptions}.(1). 
\end{proof}

\subsubsection{Canonical periods}
We fix an isomorphism $\iota : \overline{\mathbb{Q}}_p \simeq \mathbb{C}$.
Let $f \in S_k(\Gamma_1(N), \psi)$ and $\overline{\rho} = \overline{\rho}_f$ be the mod $p$ representation associated to $f$. We assume the following property.
\begin{assu} \label{assu:gorenstein}
The localized Hecke algebra at the maximal ideal corresponding to $\overline{\rho}$ is Gorenstein.
\end{assu}
\begin{thm}
Assume that $2 \leq k <p$.
If
\begin{enumerate}
\item $p$ does not divide $N$, or 
\item $p$ exactly divides $N$ and $\overline{\rho}$ is $p$-distinguished,
\end{enumerate}
then Assumption \ref{assu:gorenstein} holds.
\end{thm}
\begin{proof}
See \cite[Theorem 1.13]{vatsal-cong}.
\end{proof}
Under Assumption \ref{assu:gorenstein}, we can add a canonical $p$-integral structure on the Shimura isomorphism (\ref{eqn:shimura_isom}) following \cite[$\S$3]{vatsal-integralperiods-2013}
\[
\xymatrix@R=1em{
S(f) \otimes \mathbb{Q}_p \otimes_{\iota} \mathbb{C} \ar[r]^-{\mathrm{per}_f} & V_{\mathcal{O}_\lambda} (f) \otimes_{\mathcal{O}_\lambda} F_\lambda \otimes_{\iota} \mathbb{C} \\
\mathcal{O}_\lambda \omega_f \ar[r]^-{\mathrm{per}_f} \ar@{^{(}->}[u] & \Omega^+_f \cdot V_{\mathcal{O}_\lambda} (f)^+ + \Omega^-_f \cdot V_{\mathcal{O}_\lambda} (f)^- \ar@{^{(}->}[u]
}
\]
where $\omega_f$ is an $\mathcal{O}_\lambda$-basis of $S(f) \otimes \mathbb{Q}_p$ and it corresponds to the \emph{normalized} newform $f$. Here, $\Omega^\pm_f$ are the \textbf{integral canonical periods of $f$}.
Under the period map, we have
\begin{equation} \label{eqn:canonical_periods}
\mathrm{per}_f (\omega_f) = \Omega^+_f \cdot \gamma^+ + \Omega^-_f \cdot \gamma^-
\end{equation}
for some non-zero $\gamma^{\pm} \in V_{\mathcal{O}_\lambda} (f)^{\pm}$.
By making the same choice of $\gamma$ as Choice \ref{choice:gamma} with Equation (\ref{eqn:canonical_periods}), we obtain an ``optimal" Euler system $z^{(Np)}_{\gamma, m}(k-r)$.

\bibliographystyle{amsalpha}
\bibliography{kim-nakamura}

\end{document}